\documentclass[titlepage,a4paper,11pt]{amsart} 
\usepackage[foot]{amsaddr}
\usepackage{amssymb}
\usepackage[lite]{amsrefs}
\usepackage{t1enc}
\usepackage[latin1]{inputenc}
\usepackage[english]{babel}
\usepackage{mathrsfs} 
\usepackage{hyperref}
\usepackage[all]{xy} 
\usepackage[lite]{amsrefs}
\usepackage[left=3cm,top=3cm,right=3cm,bottom=3cm]{geometry}

\newtheorem{thm}{Theorem}[section]
\newtheorem*{thm*}{Theorem}
\newtheorem{prop}[thm]{Proposition}
\newtheorem{lem}[thm]{Lemma}
\newtheorem{cor}[thm]{Corollary}
\newtheorem{defn}{Definition}
\newtheorem{rem}[thm]{Remark}
\newtheorem{ex}[thm]{Example}

\DeclareMathOperator\Le{L}
\DeclareMathOperator\Gr{Gr}

\DeclareMathOperator\Tr{Tr}

\DeclareMathOperator\spa{span}
\DeclareMathOperator\spec{spec}
\DeclareMathOperator\A{\mathcal A}
\DeclareMathOperator\M{\mathcal M}
\DeclareMathOperator\Nv{\mathcal N}
\DeclareMathOperator\Lc{\mathcal L}
\DeclareMathOperator\Hi{\mathcal H}
\DeclareMathOperator\B{\mathcal B}
\DeclareMathOperator\K{\mathcal K}
\DeclareMathOperator\oo{{\mathcal O}}
\DeclareMathOperator\C{\mathbb{C}}
\DeclareMathOperator\R{\mathbb{R}}
\DeclareMathOperator\N{\mathbb{N}}
\DeclareMathOperator\Z{\mathbb{Z}}
\DeclareMathOperator\SOT{SOT}

\DeclareMathOperator\CAT{CAT}
\DeclareMathOperator\diag{diag}

\DeclareMathOperator\conv{conv}

\DeclareMathOperator\radius{radius}

\DeclareMathOperator\id{\mathbf{1}}

\hypersetup{
  colorlinks   = true, 
  urlcolor     = blue, 
  linkcolor    = blue, 
  citecolor   = red    
}
\title[Geometry of infinite dimensional unitary groups]{Geometry of infinite dimensional unitary groups: convexity and fixed points}
\date{\today}
\author{Martin Miglioli}
\email[Martin Miglioli]{martin.miglioli@gmail.com}
\address[Martin Miglioli]{Instituto Argentino de Matem\'atica-CONICET. Saavedra 15, Piso 3, (1083) Buenos Aires, Argentina}
\thanks{The author was supported by IAM-CONICET, grants PIP 2010-0757 (CONICET) and PICT 2010-2478 (ANPCyT)}

\begin{document}
\begin{abstract}
In this article we study convexity properties of distance functions in infinite dimensional Finsler unitary groups, such as the full unitary group, the unitary Schatten perturbations of the identity and unitary groups of finite von Neumann algebras. The Finsler structures are defined by translation of different norms on the tangent space at the identity. We first prove a convexity result for the metric derived from the operator norm on the full unitary group. We also prove strong convexity results for the squared metrics in Hilbert-Schmidt unitary groups and unitary groups of finite von Neumann algebras. In both cases the tangent spaces are endowed with an inner product defined with a trace. These results are applied to fixed point properties and to quantitative metric bounds in certain rigidity problems. Radius bounds for all convexity and fixed point results are shown to be optimal. \\

\medskip

\noindent \textbf{Keywords.} unitary groups, classical Banach-Lie group, p-Schatten class, finite von Neumann algebra, p-norm, Finsler metric, path metric space, short geodesic, geodesic convexity, circumcenter, rigidity   
\end{abstract}

\maketitle
\tableofcontents




\section{Introduction}
In this article we study convexity properties of distance functions in Finsler unitary groups, where the Finsler structures are defined by translation of different norms on the tangent space at the identity. As in the finite dimensional setting, an infinite dimensional manifold with a Finsler metric (a continuous distribution of norms in the tangent spaces) becomes a metric space. The distance between two points is given by the infimum of the lengths of the smooth curves which join these points. In \cites{pschatten,fredholm,finitemeasure} metric properties of unitary groups of Schatten perturbations of the identity and finite von Neumann algebras were studied. We generalize some results of these articles and we apply them to fixed point properties.   

Let $U$ stand for the unitary group of an infinite dimensional separable Hilbert space. We prove the following convexity result. Denote by $d_\infty$ the distance induced by the Finsler metric given by the operator norm. If $u,v,w\in U$ satisfy $d_\infty(u,w),d_\infty(v,w)\leq \pi/2$ and $d_\infty(u,v)<\pi$, then for the geodesic $\gamma$ joining $u$ and $v$ in $U$ the map $f(s)=d_\infty(\gamma(s),w)$ is convex for $s\in [0,1]$. In particular the convexity radius of the geodesic balls in $U$ is $\pi/2$. In order to obtain this result, we first prove that the maximum radius of convexity of the $d_\infty$ balls in the unitary compact perturbation of the identity is $\pi/2$. This is done with a continuous induction argument using results from \cite{fredholm}. Then approximations in the strong operator topology are used to prove the convexity of the $d_\infty$ distance in the full unitary group, and the $\pi/2$ bound for convexity is shown to be optimal. These results are important to get sharper version of previous results and several corollaries. A corollary is the following: let $u$ be a unitary operator and let $e^{tx}$ be a uniformly continuous one-parameter group of unitaries and consider 
$$\spec(-i\log(ue^{tx})).$$
If in an interval for the parameter $t$ the difference between the maximum and the minimum of this spectrum is less than $\pi$ then the maximum of the spectrum is convex and the minimum of the spectrum is concave in $t$. An analogous statement in the linear case of self-adjoint operators is immediate from the convexity of the norm: for self-adjoint operators $x$ and $y$ the maximum and minimum of the spectrum of $x+ty$ are respectively convex and concave in $t$. In the non linear case of positive invertible operators an analogous statement was proved in \cite{corach}.  

We also address convexity properties of squared distance functions. When the groups consist of unitaries of finite von Neumann algebras and of Hilbert-Schmidt perturbations of the identity the tangent spaces are endowed with inner products defined by the trace. The Hilbert-Schmidt unitary group is an infinite dimensional Riemannian manifold. The unitary group a finite von Neumann algebra has incomplete tangent spaces, but a metric $d_2$ can be defined in both cases.  We establish strong convexity results for the functions $d_2(w,\cdot)^2$ in closed balls $B_\infty[w,r]$ defined with the $d_\infty$ distance. More specifically, if for a unit speed geodesic $\beta$ and an $r<\pi/2$ we have $d_\infty(w,\beta(t))\leq r$, then the function $f(t)=d_2(w,\beta(t))^2$ satisfies 
$$f''\geq \sin(2r)/r.$$
In the case of the Hilbert-Schmidt groups the ratio $d_\infty/d_2$ can be arbitrarily small, so restricting the domain for strong convexity to $d_\infty$-balls is much less restrictive. In the case of finite von Neumann algebras the ratio $d_\infty/d_2$ can be arbitrarily large, so there is no convexity of $d_2(w,\cdot)^2$ in $d_2$-balls.

These strong convexity results for $d_2$ metrics are applied to the existence of analogues of circumcenters. In simpler contexts the existence of circumcenters follows from the strong convexity properties of the function
$$f_A(v)= \sup_{a\in A}d_2(a,v)^2$$
which imply that a unique minimizer exists, which is the circumcenter of the set $A$. In the context of our article we first define a set 
$$C=\bigcap_{a\in A}B_\infty [a,r]$$ 
of approximate circumcenters for an $r<\pi/2$, where $B_\infty[a,r]$ are closed balls with respect to the $d_\infty$ metric. Then in the set $C$ we find a minimizer of the function $f_A$. Hence the circumradius condition for the existence of of minimizers of $f_A$ is in term of the $d_\infty$ metric, which is less restrictive as was remarked above. From these results fixed point properties for group actions follow. The $d_\infty$-bounds on the circumradius of orbits for the existence of fixed points are shown to be optimal. Examples of fixed points are equivalences of representations and invariant subspaces of representations. We also give several examples of geodesically convex subspaces, since the aforementioned results hold in these subspaces.

The contents of the paper are as follows. In Section \ref{prel} we review several results on the metric geometry of unitary groups endowed with Finsler structures derived from different operator norms. These are the full unitary group, groups of unitary Schatten perturbations of the identity, and unitary groups of finite von Neumann algebras. In Section \ref{sprop} we first establish convexity properties for the geodesic distance in the full unitary group. In Section \ref{slambdaconv} we obtain strong convexity properties of the squared $d_2$ distance in the context of Hilbert-Schmidt perturbations of the identity and finite von Neumann algebras. In Section \ref{sclosedgeo} we give several examples of closed geodesic spaces, since the results of the last section are stated in terms of these spaces. Finally, in Section \ref{scirc} we establish the existence of an analogue of circumcenter for subsets of unitary groups with $d_\infty$ circumradius less than $\pi/2$. We apply the existence of circumcenters to prove fixed point theorems with optimal $d_\infty$-bounds on the diameter of orbits. The fixed point theorems are applied to metrical bounds in certain rigidity results for representations.  

The results in this article are generalizations to the infinite dimensional context of results in the unpublished manuscript \cite{miglioli} by the same author. 

\section{Preliminaries}\label{prel}

In this section we recall several results about the metric geometry of spaces of unitaries endowed with a bi-invariant Finsler metric derived from norms on the Lie algebra. Throughout this article, we will  use the metric and the geodesic structure of these spaces, which were studied in \cites{pschatten,fredholm,finitemeasure}. 

\subsection{Full unitary group}

Let $\Hi$ be a separable Hilbert space and let  
$$U=\{u\in \B(\Hi):u^*u=uu^*=\id\}$$
be the group of unitaries. We denote with
$$\B(\Hi)_{ah}=\{x\in \B(\Hi):x^*=-x\}$$
the real Banach space of skew-adjoint operators, which is the tangent space at the identity $\id$ of $U$. On $\B(\Hi)_{ah}$ the operator norm of $x$ is
$$\|x\|_\infty=\sup_{\xi\in \Hi}\frac{\|x\xi\|}{\|\xi\|}.$$
This norm is invariant by conjugation by unitaries, so by right or left translation we can define a norm on the tangent spaces at all points of $U$. We this Finsler structure we can define a metric on $U$. Let $\Le_{\infty}$ denote the length functional for piecewise smooth curves $\alpha$ in $U$ measured with the $\|\cdot\|_\infty$ norm
$$\Le_\infty(\alpha)=\int_{t_0}^{t_1}\|\dot{\alpha}(t)\|_\infty dt.$$
We define $d_\infty(u,v)= \inf\{\Le_\infty(\gamma):\gamma \subseteq U \mbox{ joins } u \mbox{ and }v\}$ as the rectifiable distance between $u$ and $v$ in $U$. This metric is invariant by left and right translations. We next recall Proposition 5.2. of \cite{esteban}. 

\begin{prop}\label{minimalf}
Let $u\in U$ and $x\in\B(\Hi)_{ah}$ with $\|x\|_\infty\leq\pi$. Then the smooth curve $\mu(t)= ue^{tx}$ has minimal length along its path, for all $t\in[-1,1]$. Any pair of unitaries $u,v\in U$ can be joined by such a curve.
\end{prop}

We shall denote the closed metric balls with $B_\infty[w,r]=\{u\in U:d_\infty(u,w)\leq r\}$. 

\begin{rem}
The curves in this proposition are not the unique length minimizers. To see this consider an abelian subalgebra $\A$ and its unitary group $U_{\A}$. In this unitary group the geometry of $B_\infty[\id,\pi/2]$ is the same as $\{x\in \A_{ah}:\|x\|_\infty\leq \pi/2\}$ endowed with the metric defined by the uniform norm.
\end{rem}

\begin{prop}\label{basicof}
In $(U,d_\infty)$ the following holds
\begin{enumerate}
\item For $r\leq \pi$ we have $B_\infty[\id,r]=\{u\in U:\spec(u)\subseteq \exp(i[-r,r])\}$.

\item For $x\in \B(\Hi)_{ah}$ with $\|x\|_\infty\leq\pi$
$$\|\id-e^x\|_\infty=2\sin\left(\frac{\|x\|_\infty}{2}\right).$$
\item For $u,v\in U$ we have $\|u-v\|_\infty<\sqrt{2}$ if and only if $d_\infty(u,v)<\pi/2$, and $\|u-v\|_\infty<2$ if and only $d_\infty(u,v)<\pi$.
Also,
$$\frac{2}{\pi} d_\infty(u,v)\leq \|u-v\|_\infty\leq d_\infty(u,v),$$
in particular the metric space $(U,d_\infty)$ is complete.

\end{enumerate}
\end{prop}

\begin{proof}
Proposition \ref{minimalf} implies $B_\infty[\id,r]=\exp(\{x\in\B(\Hi)_{ah}:\|x\|_\infty\leq r\}$, which is equivalent to the first item. The second item follows by computing distances from $o\in\C$ to the spectrum and noting that $\spec(\id-e^x)=1-\exp(\spec(x))$. Since $\|u-v\|_\infty=\|\id-u^{-1}v\|_\infty$ and $d_\infty(u,v)=d_\infty(\id,u^{-1}v)$, we set $u^{-1}v=e^x$ with $x\in\B(\Hi)_{ah}$ such that $\|x\|_\infty \leq \pi$. The third item is implied by the second item.
\end{proof}

\begin{defn}\label{defgeod}
For $u,v\in U$ such that $d_\infty(u,v)<\pi$ we define the geodesic joining $u$ and $v$ as $\gamma_{u,v}:[0,1]\to U$,
$$\gamma_{u,v}(t)=u\exp(t\log(u^{-1}v))$$
for $t\in [0,1]$, where $\log$ is the principal branch of the natural logarithm.
\end{defn}

\begin{prop}\label{contendf}
Let $r<\pi/2$ and let $w\in U$. Then for $t\in [0,1]$ the map 
$$(u,v)\mapsto \gamma_{u,v}(t)$$
is continuous on $(B_\infty[w,r],d_\infty)\times (B_\infty[w,r],d_\infty)$. 
\end{prop}

\begin{proof}
If $u_1,u_2\in B_\infty[w,r]$ then $w^{-1}u_1,w^{-1}u_2\in B_\infty[\id,r]$. By the triangle inequality we get 
$$u_1^{-1}u_2=(w^{-1}u_1)^{-1}(w^{-1}u_2)\in B_\infty[\id,2r]=\{u\in U:\spec(u)\subseteq \exp(i[-2r,2r])\}.$$
On $C=\{u\in U:\spec(u)\subseteq \exp(i[-2r,2r])\}$ the natural branch of the logarithm is defined and the function 
$$B_\infty[\id,2r]\to \B(\Hi)_{ah},\qquad u\mapsto \log(u)$$ 
is continuous in the $\|\cdot\|_\infty$-norm. This is proved by approximating $\log\vert_C$ with polynomials $p(z,\overline{z})$ and applying the continuous functional calculus. By Proposition \ref{basicof} convergence in $d_\infty$ is equivalent to convergence in the $\|\cdot\|_\infty$-norm, so the conclusion follows.  
\end{proof}

If $\A\subseteq \B(\Hi)$ is a $C^*$-algebra let $U_{\A}$ denote the unitary group of $\A$. If $u,v\in U_{\A}$ and $d_{\infty}(u,v)<\pi$ then $\gamma_{u,v}\subseteq U_{\A}$ by stability properties of the functional calculus. Hence the $d_{\infty}$ distance in $U_{\A}$ agrees with the $d_{\infty}$ distance of $U$.

\subsection{Unitary Schatten perturbation of the identity}

Denote the $p$-Schatten class by 
$$\B_p(\Hi)=\{x\in\B(\Hi): \Tr((x^*x)^{\frac{p}{2}})<\infty\},$$ 
where $\Tr$ is the usual trace in $\B(\Hi)$. We shall focus on the case when $p$ is an even integer. The spaces $\B_p(\Hi)$ are Banach spaces with the norms $\|x\|_p=\Tr((x^*x)^{\frac{p}{2}})^\frac{1}{p}.$
The set of  skew-hermitian operators in $\B_p(\Hi)$ is $\B_p(\Hi)_{ah}=\{x\in \B_p(\Hi): x^*=-x\}$. Consider the following classical Banach-Lie group of operators:
$$U_p=\{u\in U: u -\id\in\B_p(\Hi)\}.$$ 
These groups have differentiable structure when endowed with the metric $\|u-v\|_p$ and the Banach-Lie algebra of $U_p(\Hi)$ is the (real) Banach space $\B_p(\Hi)_{ah}$, see \cite{harpe}. Note that the exponential map is a bijection between the sets
$$ \{z\in \B_p(\Hi)_{ah}: \|z\|_\infty < \pi\}\to\{u\in U_p:\|\id - u\|_\infty < 2\}.$$
Moreover, $\exp: \{z\in \B_p(\Hi)_{ah}: \|z\|_\infty \leq \pi\}\to U_p$ is surjective. We introduce a Finsler metric on $U_p$ as follows. Let $\Le_p$ denote the length functional for piecewise smooth curves $\alpha$ in $U_p$ measured with the $\|\cdot\|_p$ norm:
$$\Le_p(\alpha)=\int_{t_0}^{t_1}\|\dot{\alpha}(t)\|_p dt,$$
and we define $d_p(u,v)= \inf\{\Le_p(\gamma):\gamma \subseteq U \mbox{ joins } u \mbox{ and }v\}$ as the rectifiable distance between $u$ and $v$ in $U_p$. This metric is invariant by left and right translations.

We recall Theorem 3.2 in \cite{pschatten} concerning the minimality of geodesics in $(U_p,d_p)$. This theorem could be derived from the general theory of Riemannian manifolds in the case $p = 2$. 

\begin{thm}\label{geodesicasp}
Let $p$ be an even integer and consider $(U_p,d_p)$ as defined above. The following facts hold:
\begin{enumerate}
 
\item Let $u\in U_p$ and $x\in \B_p(\Hi)_{ah}$ with $\|x\|_\infty\leq \pi$. Then the curve $\mu(t)=ue^{tx}$, $t\in [0,1]$, is shorter than any other piecewise smooth curve in $U_p$ joining the same endpoints. Moreover, if $\|x\|_\infty< \pi$ then $\mu$ is unique with this property.

\item Let $u,v\in U_p$. Then there exists a minimal geodesic curve joining them. If $\|u-v\|_\infty<2$ then this geodesic is unique. There are in $U_p$ minimal geodesics of arbitrary length, thus the diameter of $U_p$ is infinite.

\item If $u,v\in U_p$, then 
 $$\left(\sqrt{1-\frac{\pi^2}{12}}\right)d_p(u,v)\leq \|u-v\|_p\leq d_p(u, v).$$
In particular the metric space $(U,d_p)$ is complete.
\end{enumerate}
\end{thm}

\begin{thm}[{\cite[Theorem 3.6]{pschatten}}]\label{strictp}
Let $p$ be an even integer, $u\in U_p$, and let $\beta\subseteq B_p(u,\pi/2)$ be a non constant geodesic. Then the function 
$$f(t)=d_p(u,\beta(t))^p$$ 
is strictly convex.
\end{thm}

Let $\K\subseteq \B(\Hi)$ denote the ideal of compact operators. We define 
$$U_c=\{u\in U:u-\id\in \K\},$$
and note that $U_c=\exp(\K_{ah})$. The following theorem was proved using Theorem \ref{strictp} in \cite{fredholm}. 

\begin{thm}[{\cite[Theorem 2.8]{fredholm}}]\label{conv}
Let $u\in U_c$ and let $\beta\subseteq B_\infty(u,\pi/2)$ be a geodesic. Then the function 
$$f(t)=d_\infty(u,\beta(t))$$ 
is convex.
\end{thm}

\begin{rem}
Finite dimensional unitary groups are examples of subgroups of $U_p$ or $U_c$. If $\Hi_1\subseteq \Hi$ is finite dimensional subspace we can define an embedding of $U(\Hi_1)$ into $U_p$ or $U_c$ given by 

$$u\mapsto\left(\begin{array}{cc}
u & 0 \\
\\
0 & \id
\end{array}\right)$$
for $u\in U(\Hi_1)$.
\end{rem}

\subsection{Unitary group of a finite von Neumann algebra}

Let $\M$ be a von Neumann algebra with a finite and faithful trace $\tau$. Denote by $U_{\M}$ the group of invertible unitary operators in $\M$. For $1 < p <\infty$ the space $\M$ is endowed with the norm 
$$\|x\|_p=\tau((x^* x)^{\frac{p}{2}})^{\frac{1}{p}}.$$ 
When $p = 2$, $\M$ is a pre-Hilbert space with the inner product $\langle x, y \rangle_{\tau}= \tau(y^*x)$.

A Finsler metric on $U$ is defined as follows. Let $L_p$ denote the length functional for piecewise smooth curves in $U_{\M}$, measured with the $p$-norm
$$L_p=\int_{t_0}^{t_1} \|\dot{\alpha}(t)\|_p dt.$$
Smooth means $C^1$ and with nonzero derivative, relative to the uniform topology of $\M$. We define $d_p(u,v)= \inf\{\Le_p(\gamma):\gamma \subseteq U \mbox{ joins } u \mbox{ and }v\}$ as the rectifiable distance between $u$ and $v$ in $U$. This metric is invariant by left and right translations.

We next recall Theorem 3.5 in \cite{finitemeasure} which collects several results concerning the rectifiable $d_p$ distance in the unitary group of $\M$, such as: minimality of geodesics, uniqueness of such geodesics, comparison with the usual $d_p$ distance, and  a fundamental convexity result.

\begin{thm}\label{basicvn}
Let $2\leq p<\infty$. The following facts hold
\begin{enumerate}
\item Let $u\in U_{\M}$ and $x \in \M_{ah}$ with $\|x\|_\infty\leq\pi$. Then the curve $\mu(t)=ue^{tx}$, $t\in [0, 1]$ is shorter than any other smooth curve in $U_{\M}$ joining the same endpoints, when we measure them with the length functional $\Le_p$. Moreover, if $\|x\|_\infty < \pi$, the curve $\mu$ is unique with this property among all the $C^2$ curves in $U_{\M}$.
\item Let $u, v\in U_{\M}$. Then there exists a minimal geodesic curve joining them. If
$\|u -v\| < 2$, this geodesic is unique among all the $C^2$ curves.
\item The diameter of $U_{\M}$ is $\pi$ for all the $p$-norms. If $u, v\in U_{\M}$ then
$$\left(\sqrt{1-\frac{\pi^2}{12}}\right)d_p(u,v)\leq \|u-v\|_p\leq d_p(u, v).$$
In particular the metric space $(U_{\M},d_p)$ is complete.

\item Let $p$ be an even positive number, $u, v, w \in U_{\M}$, with
$$\|u - v\| <\sqrt{2},\quad \|w - v\| <\sqrt{2}-\|u - v\|.$$
Let $\beta$ be a short geodesic joining $v$ to $w$ in $U_{\M}$. Then  
$$f(t)=d_p(u,\beta(t))^p$$ 
is a strictly convex function, provided $u$ does not belong to any prolongation of $\beta$.
\end{enumerate}
\end{thm}

\begin{rem}
Finite dimensional unitary groups are examples of subgroups of $U_{\M}$. From a system of matrix units we can define an inclusion $M_n(\C)\subseteq \M$ of a matrix algebra, and therefore an inclusion $U_{M_n(\C)}\subseteq U_{\M}$ of unitary groups.
\end{rem}

\section{Convexity properties of $d_{\infty}$}\label{sprop}

In this section we obtain better bounds for the convexity of balls in $(U_c,d_\infty)$ stated in Proposition \ref{geoconv}. This result is important to prove the convexity of the $d_\infty(\cdot,u)$ in the full unitary group, with optimal bounds. This is achieved through approximations in the strong operator topology ($\SOT$).

\subsection{Optimal radius of convexity}

We obtain $\pi/2$ as optimal radius for the convexity of balls in $(U_c,d_\infty)$ using a continuous induction argument.

\begin{prop}\label{geoconv}
Let $w\in U_c$, then for each $r<\pi/2$ the set $B_\infty[w,r]$ is geodesically convex, that is, if $u,v\in B_\infty[w,r]$ then $\gamma_{u,v}\subseteq B_\infty[u,r]$. Also, if $u,v\in B_\infty[u,\pi/2]$ and $d_\infty(u,v)<\pi$, then $\gamma_{u,v}\subseteq B_\infty[u,\pi/2]$.
\end{prop}

\begin{proof}
Since left translation is isometric we can assume that $w=\id$. We will repeatedly use the continuity in the $d_\infty$ metric of $(u,v)\mapsto \gamma_{u,v}(t)$, see Proposition \ref{contendf}.

For $r<\pi/2$ define the set 
$$C=\{(u,v)\in B_\infty(\id,r)\times B_\infty(\id,r):\gamma_{u,v}\subseteq B_\infty(\id,r)\}.$$
Note that $B_\infty(\id,r)\times B_\infty(\id,r)$ is connected since 
$$B_\infty(\id,r)=\exp(\{x\in\K_{ah}:\|x\|_\infty<r\})$$ 
is connected. We will show that $C$ is not empty, open and closed in the space $B_\infty(\id,r)\times B_\infty(\id,r)$, and is therefore equal to this space. If $u,v\in B_\infty(\id,r/2)$, then by the triangle inequality $\gamma_{u,v}\subseteq B_\infty(\id,r)$, so $C$ is not empty.

The set $C$ is open, otherwise there are $(u_n,v_n)_n \to (u,v)$ in $B_\infty(\id,r)\times B_\infty(\id,r)$, with $(u,v)\in C$, and $(t_n)_n\subseteq [0,1]$ such that $\gamma_{u_n,v_n}(t_n)\notin B_\infty(\id,r)$. We choose a convergent subsequence $t_{n_m}\to t'$, so that $\gamma_{u_{n_m},v_{n_m}}(t_{n_m})\to \gamma_{u,v}(t')\notin B_\infty(\id,r)$. This is a contradiction, since $(u,v)\in C$.

The set $C$ is closed, otherwise let $(u_n,v_n)_n$ be a sequence in $C$ which converges to $(u,v)\in (B_\infty(\id,r)\times B_\infty(\id,r))\cap C^c$. For all $t\in [0,1]$ we have $\gamma_{u_{n},v_{n}}(t)\to \gamma_{u,v}(t)$ and $d_\infty(\id,\gamma_{u_{n_m},v_{n_m}}(t))<r$ so that $d_\infty(\id,\gamma_{u,v}(t))\leq r$. Suppose there is a $t\in [0,1]$ such that  $d_\infty(\id,\gamma_{u,v}(t))= r$. Then the convexity of the map $t\mapsto d_\infty(\id,\gamma_{u,v}(t))$ asserted in Theorem \ref{conv} is not satisfied.

Therefore $C=B_\infty(\id,r)\times B_\infty(\id,r)$, that is $B_\infty(\id,r)$ is geodesically convex for $r<\pi$. If $u,v\in B_\infty[\id,r]$ for $r<\pi/2$ then $\gamma_{u,v}\subseteq B_\infty(\id,r')$ for all $r'$ such that $r<r'<\pi/2$. Taking $r'\to r$ we conclude that $\gamma_{u,v}\subseteq B_\infty[\id,r]$.

To prove the last statement assume that $u,v\in B_\infty[\id,\pi/2]$ and $d_\infty(u,v)<\pi$. We can chose sequences $u_n\to u$ and $v_n\to v$ such that $d_\infty(u_n,v_n)<\pi$ and $u_n,v_n\in B_\infty[\id,r_n]$, with $r_n<\pi/2$ and $r_n\to \pi/2$. This can be done as follows, if $u=e^x$ and $v=e^y$ we take $u_n=e^{(1-\frac{1}{n})x}$ and $v_n=e^{(1-\frac{1}{n})y}$ with $r_n=(1-\frac{1}{n})\frac{\pi}{2}$. Then by the first part of the proposition $\gamma_{u_n,v_n}\subseteq B_\infty[\id,r_n]\subseteq B_\infty[\id,\pi/2]$. Since the geodesics depend continuously on the endpoints in the limit $\gamma_{u,v}(t)\in B_\infty[\id,\pi/2]$ for $t\in [0,1]$.
\end{proof}

The $\pi/2$ bound for the radius of convexity is optimal.  We embed the finite dimensional unitary group $U(\C)$ in $U_c$ and consider the following example. 

\begin{ex}
We take a ball of center $1$ and radius $\pi/2+\epsilon$ in the group $U(\C)\simeq S^1$. The geodesic $\gamma(t)=e^{ti}$ for $t\in [-3/2 \pi+\frac{1}{2}\epsilon,-\pi/2-\frac{1}{2}\epsilon]$ connects two points inside the ball, but it is not contained in the ball.
\end{ex}

With this proposition we can sharpen \cite[Theorem 3.6]{pschatten} using Remark 3.7 in that article. 

\begin{thm}\label{strict}
Let $p$ be an even integer, $w$ an element of the group of unitaries $U_p$, and let $u,v\in B_\infty(u,\pi/2)\cap U_p$ such that $u\neq v$. Let $\beta=\gamma_{u,v}$ be a the non constant geodesic joining $u$ and $v$. Then the function 
$$f(t)=d_p(u,\beta(t))^p$$ 
is strictly convex.
\end{thm}

\begin{proof}
By Remark 3.7 in \cite{pschatten} the strict convexity holds as long as $d_\infty(\beta(t),\id)=\|w_t\|_\infty<\pi/2$, where $w_t=\log(\beta(t))$. By the convexity established in Proposition \ref{geoconv} we know that $d_\infty(\beta(s),\id)\leq \max\{d_\infty(\beta(0),\id),d_\infty(\beta(1),\id)\}=\max\{d_\infty(u,\id),d_\infty(v,\id)\}<\pi/2$.

Theorem 3.6 in \cite{pschatten} does not include the case when a prolongation of $\beta$ includes $u$. Note that if a prolongation of $\beta$ includes $u$ then the prolongation is of the form $\beta(t)=ue^{tx}$ and 
$$f(t)=\|x\|_p^pt^p.$$
Therefore $f''(t)=p(p-1)\|x\|_p^{p}t^{p-2}$, and $f''(t)=0$ only for $t=0$ and $p\geq 4$ even. Hence strict convexity still holds in this case.
\end{proof}

\subsection{Convexity of the distance in the full unitary group}
We prove that the distance $d_\infty$ is convex in the full unitary group. The following lemma is used to prove that closed balls in the $d_\infty$ distance are $\SOT$ closed.

\begin{lem}\label{bolpos}
For $r\leq \pi$ we have
$$ B_\infty[\id,r]=\{u\in U:u+u^*\geq 2\cos(r)\id\}.$$
\end{lem}

\begin{proof}
This follows from the equality
$$B_\infty[\id,r]=\{u\in U:\spec(u)\subseteq \exp(i[-r,r])\}$$
stated in Proposition \ref{basicof} and the following property of the numerical range. An operator $a$ has numerical range in the half-space $\{b_1+b_2 i:b_1\geq 0\}$ if and only if $a+a^*\geq 0$. Hence for $c\in\R$ an operator $a+c\id$ has numerical range in the half-space $\{b_1+b_2 i:b_1\geq 0\}$ if and only if $a+a^*+2c\id\geq 0$. Hence an operator $a$ has numerical range in the half-space $\{b_1+b_2 i:b_1\geq d\}$ if and only if  $a+a^*\geq 2d\id$.
\end{proof}

\begin{rem}
From translation invariance and the triangle inequality 
$$d_\infty(uv,\id)=d_\infty(u,v^{-1})\leq d_\infty(u,\id)+d_\infty(\id,v^{-1})=d_\infty(u,\id)+d_\infty(v,\id).$$ 
Hence if $r_1,r_2\leq \pi/2$, $u+u^{-1}\geq 2\cos(r_1)\id$ and $v+v^{-1}\geq 2\cos(r_1)\id$ then 
$$uv+(uv)^{-1}\geq 2\cos(r_1+r_2)\id.$$
\end{rem}

\begin{prop}\label{ballcwot}
For $0\leq r\leq \pi$ and $u\in U$ the ball $B_\infty[u,r]$ is closed in the strong operator topology ($\SOT$).
\end{prop}

\begin{proof}
Let $(v_n)_n\subseteq B_\infty[u,r]$ and $v_n\to v$ in $\SOT$ in $U$. By translation invariance of the metric $(u^{-1}v_n)_n\subseteq B_\infty[\id,r]$. Since multiplication is $\SOT$ continuous $u^{-1}v_n\to u^{-1}v$ in $\SOT$.  By Lemma \ref{bolpos} $u^{-1}v_n+(u^{-1}v_n)^*\geq 2\cos(r)\id$. Since by Proposition 4.1 Chapter II in \cite{takesaki} the adjoint operation is $\SOT$-continuous in $U$ we get $u^{-1}v_n+(u^{-1}v_n)^*\to u^{-1}v+(u^{-1}v)^*$ in $\SOT$. It follows that $u^{-1}v+(u^{-1}v)^*\geq 2\cos(r)\id$. Therefore by Lemma \ref{bolpos} $u^{-1}v\in B_\infty[\id,r]$, which is equivalent to $v\in B_\infty[u,r]$.
\end{proof}

We recall some results about the continuity of the functional calculus in the $\SOT$-topology. 

\begin{defn}
For a compact set $G\subseteq \C$ we define
$$\Lc_G=\{a\in \B(\Hi):a\mbox{  is normal and  }\spec(a)\subseteq G\}.$$
\end{defn}

The following is a special case of Theorem 4.7 Chapter II in \cite{takesaki}.

\begin{thm}\label{softfunc}
For a compact set $G\subseteq \C$ and a continuous function $f:G \to \C$ the map 
$$\Lc_G\to\Lc_{\C}\mbox{  given by  } a\mapsto f(a)$$
is continuous in $\SOT$.
\end{thm}

\begin{prop}\label{contendsot}
For $0\leq r<\pi/2$, $t\in\R$ and $w\in U$ let the map $B_\infty[w,r]\times B_\infty[w,r]\to U$ be defined by 
$$(u,v)\mapsto \gamma_{u,v}(t)=u\exp(t\log(u^{-1}v)).$$
Then this map is $\SOT$ continuous.
\end{prop}

\begin{proof}
Fix a $t\in \R$ and assume $w=\id$ since left multiplication is isometric and $\SOT$-continuous. Let $S_r=\{\exp(si)\in \C:-r\leq s\leq r\}$ and let $I_r=\{si \in\C:-r\leq s\leq r\}$. Then for $0\leq r\leq \pi$ we have 
$$\Lc_{S_r}=B_\infty[\id,r]\mbox{  and  }\Lc_{I_r}=\{x\in \B(\Hi)_{sa}:\|x\|_\infty\leq r\}.$$
Let $(u_n)_n\subseteq B_\infty[\id,r]$ such that $u_n\to u\in B_\infty[\id,r]$ in $\SOT$ and let $(v_n)_n\subseteq B_\infty[\id,r]$ such that $v_n\to u\in B_\infty[\id,r]$ in $\SOT$. We have $u_n^{-1}v_n\to u^{-1}v$ in $\SOT$. Since $u_n^{-1}v_n\in B_\infty[\id,r_1+r_2]=\Lc_{S_{r_1+r_2}}$ Theorem \ref{softfunc} with the function 
$$f:\Lc_{S_{r_1+r_2}}\to\C,\qquad \alpha\mapsto\exp(t\log(\alpha))$$ 
asserts that $\exp(t\log(u_n^{-1}v_n))\to\exp(t\log(u^{-1}v))$ in $\SOT$. Hence 
$$\gamma_{u_n,v_n}(t)=u_n\exp(t\log(u_n^{-1}v_n))\to u\exp(t\log(u^{-1}v))=\gamma_{u,v}(t)$$ in $\SOT$ and the proposition is proved.
\end{proof}

\begin{thm}\label{convfull}
For $0\leq r< \pi/2$ let $u,v,w\in U$ satisfy $u,v\in B_\infty[w,r]$, and let $\gamma_{u,v}$ be the geodesic joining $u$ and $v$. Then
$$f(t)=d_\infty(\gamma_{u,v}(t),w)$$
is a convex function. If $u,v\in B_\infty[w,\pi/2]$ satisfy $d_\infty(u,v)<\pi$ then the the same conclusion holds.  
\end{thm}

\begin{proof}

We assume that $w=\id$. For $r_1,r_2<\pi/2$ let $u\in B_\infty[\id,r_1]$ and $v\in B_\infty[\id,r_2]$. Since $d_\infty(\id,u),d_\infty(\id,v)<\pi/2$, by bi-invariance of the metric and by the triangle inequality we get $d_\infty(u,v)<\pi$. So there is a geodesic $\gamma_{u,v}(t)=u\exp(t\log(u^{-1}v))$ joining $u$ and $v$. We want to show that 
$$\gamma_{u,v}(t)\in B_\infty[\id,(1-t)r_1+tr_2]\mbox{  for  }t\in[0,1].$$ 
We write $u=e^x$ and $v=e^y$ with $x,y\in\B(\Hi)_{sa}$ such that $\|x\|_\infty\leq r_1$ and $\|y\|_\infty\leq r_2$. Let $(\xi_n)_n$ be an orthonormal basis of $\Hi$ and let $p_n$ be the projection onto $\spa\{\xi_1,\dots,\xi_n\}$. We define $x_n=p_nxp_n$ and $y_n=p_nyp_n$.  Since $p_n\to\id$ in $\SOT$ we conclude that $x_n\to x$ and $y_n\to y$ in $\SOT$. 

Since $x_n$ and $y_n$ are skew-adjoint we define the unitaries $u_n=e^{x_n}$ and $v_n=e^{y_n}$. Since $\|x_n\|_\infty\leq\|x\|_\infty\leq  r_1$ and $\|y_n\|_\infty\leq\|y\|_\infty\leq r_2$ for all $n\in \N$ we conclude that $(x_n)_n, x$ are in $\Lc_{I_{r_1}}$ and $(y_n)_n,y$ are in $\Lc_{I_{r_2}}$. Therefore we can apply Theorem \ref{softfunc} to these sets with the function $f=\exp$ to conclude that $u_n\to u$ and $v_n\to v$ in $\SOT$. 

For $n\in\N$ note that $u_n\in B_\infty[\id,r_1]$ and $v_n\in B_\infty[\id,r_2]$. These are unitaries which are the identity on $\spa\{\xi_1,\dots,\xi_n\}^\perp$, hence they are elements of $U_c$. By Proposition \ref{geoconv} we have $\gamma_{u_n,v_n}\subseteq B_\infty[\id,\max\{r_1,r_2\}]$. Hence Theorem \ref{conv} states that $\gamma_{u_n,v_n}(t)\in B_\infty[\id,(1-t)r_1+tr_2]$. 

Observe that $B_\infty[\id,(1-t)r_1+tr_2]=\{u\in U:u+u^*\geq 2\cos((1-t)r_1+tr_2)\}$, which by Proposition \ref{ballcwot} is $\SOT$ closed. Proposition \ref{contendsot} states that 
$$u_n\exp(t\log(u_n^{-1}v_n))\to u\exp(t\log(u^{-1}v))$$ 
in $\SOT$. Hence $\gamma_{u,v}(t)=u\exp(t\log(u^{-1}v))\in B_\infty[\id,(1-t)r_1+tr_2]$, which is what we wanted to prove.

To prove the last statement assume that $u,v\in U$ with $d_\infty(u,v)<\pi$, $d_\infty(u,\id)=r_1\leq\pi/2$ and $d_\infty(v,\id)=r_2\leq\pi/2$. We want to show that for $t\in [0,1]$ we have $d_\infty(\gamma_{u,v}(t),\id)\leq (1-t)r_1+tr_2$. We can chose sequences $u_n\to u$ and $v_n\to v$ such that $d_\infty(u_n,v_n)<\pi$ for sufficiently large $n\in\N$ and such that $d_\infty(u_n,\id)<r_1\leq\pi/2$ and $d_\infty(v_n,\id)<r_2\leq\pi/2$ for all $n\in \N$. This can be done as follows, if $u=e^x$ and $v=e^y$ we take $u_n=e^{(1-\frac{1}{n})x}$ and $v_n=e^{(1-\frac{1}{n})y}$.  By the first statement of the theorem $\gamma_{u_n,v_n}(t)\in B_\infty[\id,(1-t)r_1+tr_2]$. Since the geodesics depend continuously on the endpoints in the $d_\infty$ metric we see that in the limit $\gamma_{u,v}(t)\in B_\infty[\id,(1-t)r_1+tr_2]$, which is what we wanted to prove.
\end{proof}

\begin{cor}\label{tita}
Let $u\in U$ and $x\in\B(\Hi)_{ah}$ such that $\spec(ue^{tx})\in S^1 \backslash\{-1\}$ for $t\in (0,l)$. Define  
$$\theta_{\max}(t)=\max\spec(-i\log(ue^{tx}))\mbox{   and   }\theta_{\min}(t)=\min\spec(-i\log(ue^{tx})).$$ 
If for an open interval $I\subseteq (0,l)$ we have $\theta_{\max}(t)-\theta_{\min}(t)<\pi$ for $t\in I$, then in $I$ the function $\theta_{\max}$ is convex and the function $\theta_{\min}$ is concave. 
\end{cor}

\begin{proof}
Let $I$ be an interval as in the the statement of the corollary. If $t_0\in I$ then there is a $c\in \R$ and an open interval $J\subseteq I$ with $t_0\in J$ such that 
$$-\pi/2<\theta_{\min}(t)+c\leq \theta_{\max}(t)+c<\pi/2$$ 
and $\theta_{\max}(t)+c>\vert \theta_{\min}(t)+c \vert$ for $t\in J$. The first inequalities imply that the geodesic $\beta(t)=e^{ic\id}ue^{tx}$ lies in $B_\infty(\id,\pi/2)$ for $t\in J$ and the second inequality implies that 
$$\theta_{\max}(t)+c=d_\infty(\id,e^{ic\id}ue^{tx})$$
for $t\in J$. By Theorem \ref{conv} $\theta_{\max}(t)+c$ is convex in $J$. Since convexity is a local property we conclude that $\theta_{\max}$ is convex in $I$. An analogous argument implies the concavity of $\theta_{\min}$ in $I$. 
\end{proof}

The $\pi$ bound for the difference between the maximum and minimum eigenvalue is optimal as the following example shows.

\begin{ex}\label{ejemplomatrices}
Consider the following curve in $U(\C^2)$:  
$$ue^{tx}=
 \left(\begin{array}{cc}
e^{i\theta} & 0 \\
\\
0 & e^{-i\theta}
\end{array}\right)
\left(\begin{array}{cc}
\cos(t) & \sin(t) \\
\\
-\sin(t) & \cos(t)
\end{array}\right).
$$
The eigenvalues of $ue^{tx}$ are
$$\cos(\theta)\cos(t)\pm i\sqrt{1-(\cos(\theta)\cos(t))^2}.$$
Hence the eigenvalues of $\log(ue^{tx})$ are $\pm i\arccos(\cos(\theta)\cos(t))$. If we set $$f(t)=\arccos(\cos(\theta)\cos(t))$$ 
some calculations yield 
$$f''(0)=\cot(\theta)$$ 
for $\theta\neq 0$ and $-\pi<\theta<\pi$. Taking $\theta=\pi/2+\epsilon$ and $t\in(-\epsilon,\epsilon)$ we see that the conclusion of Corollary \ref{tita} does not hold for small $\epsilon>0$.
\end{ex}

Symmetries are unitaries $u$ such that $u^*=u$, that is, $\spec(u)\subseteq \{1,-1\}$. Their geometry was studied in \cite{portarecht} and subsequent articles, see \cite{esteban} and the reference therein. The following corollary can be proved with elementary techniques.

\begin{cor}\label{geoconvref}
If $u,v$ are symmetries such that $d_\infty(u,v)<\pi$, then the unique geodesic $\gamma_{u,v}$ consists of symmetries. 
If $\gamma_{u,v}(t)=ue^{tx}$ with $\|x\|_\infty<\pi$ then $\gamma_{u,v}(t)=e^{-\frac{1}{2} tx}ue^{\frac{1}{2} tx}$.  
\end{cor}

\begin{proof}
We note that $B_\infty[i\id,\pi/2]=\{u\in U:\spec(u)\subseteq \exp(i[0,\pi])\}$ and $B_\infty[-i\id,\pi/2]=\{u\in U:\spec(u)\subseteq \exp(i[-\pi,0])\}$. Hence 
$$B_\infty[i\id,\pi/2]\cap B_\infty[-i\id,\pi/2]=\{u\in U:\spec(u)\subseteq \{1,-1\}\}$$
is the space of symmetries. Since by Theorem \ref{convfull} $\gamma_{u,v} \subseteq B_\infty[i\id,\pi/2]$ and $\gamma_{u,v}\subseteq B_\infty[-i\id,\pi/2]$ the first assertion follows.
The geodesic $\gamma_{u,v}(t)=ue^{tx}$ consists of symmetries, therefore
$$ue^{tx}=(ue^{tx})^*=e^{-tx}u^*=e^{-tx}u.$$
If we differentiate with respect to $t$ and set $t=0$ we get $ux=-xu$. Hence $e^{-\frac{1}{2} tx}u=ue^{\frac{1}{2} tx}$, from which the second assertion follows. 
\end{proof}

Using Lemma \ref{bolpos} we can reformulate Theorem \ref{convfull}.

\begin{cor}
Let $r_1,r_2\leq \pi/2$, $v$ a unitary operator, and $x$ a skew-adjoint operator such that $\|x\|_\infty<\pi$, $v+v^{-1}\geq 2\cos(r_1)\id$ and $(ve^x)+(ve^x)^{-1}\geq 2\cos(r_2)\id$. Then 
$$(ve^{tx})+(ve^{tx})^{-1}\geq 2\cos((1-t)r_1+tr_2)\id$$
for all $t\in [0,1]$. 
\end{cor}

With Theorem \ref{convfull} we can prove a sharper version of item (4) of Theorem \ref{basicvn}.

\begin{thm}\label{strictlp}
Let $p$ be an even positive number, $u, v, w\in U_{\M}$ such that $v,w\in B_\infty[u,r]$ with $0\leq r<\pi/2$. Let $\beta$ be the short geodesic joining $v$ to $w$ in $U_{\M}$. Then 
$$f(t)=d_p(u,\beta(t))^p$$
is a strictly convex function. 
\end{thm}

\begin{proof}
Assume that $u=\id$ and that a prolongation of $\beta$ does not include $u=\id$. The assumption is that $v,w\in B_\infty[\id,r]$, so Theorem \ref{convfull} implies that $\beta\subseteq B_\infty[u,r]$, that is, $d_\infty(\beta(t),\id)\leq r<$ for all $t\in [0,1]$. Let us write $\beta(t)=e^{x_t}$ with $x_t$ skew-adjoint and $\|x_t\|_\infty\leq r$.  Since by Proposition \ref{basicof} $\|1-e^x\|_\infty=2\sin(\frac{\|x\|_\infty}{2})$ we have 
$$\|\id-\beta(t)\|_\infty=2\sin\left(\frac{\|x_t\|_\infty}{2}\right)\leq 2\sin\left(\frac{r}{2}\right)<\sqrt{2}$$
for all $t\in [0,1]$. Given a $t_0 \in [0,1]$ choose a an interval $I\subseteq [0,1]$ such that $t_0\in I$ and which is small enough so that 
$$\|\beta(t_0)-\beta(t)\|_\infty<\sqrt{2}-\|\id-\beta(t)\|_\infty\leq\sqrt{2}-2\sin\left(\frac{r}{2}\right)<0$$  
for all $t\in I$. This can be done by the continuity of $t\mapsto \beta(t)$ in the $\|\cdot\|_\infty$ norm. Then Theorem \ref{basicvn} states that $d_p(u,\beta(t))^p$ is strictly convex for $t\in I$. Since strict convexity is a local property we get strict convexity for all $t\in [0,1]$.
If a prolongation of $\beta$ includes $u$ then the prolongation is of the form $\beta(t)=ue^{tx}$. In this case 
$$d_p(u,\beta(t))^p=\|x\|^pt^p,$$
so strict convexity also holds.
\end{proof}

\section{Strong convexity of the squared $d_2$ metrics}\label{slambdaconv}

In this section we prove strong convexity properties of the function $d_2(\cdot,u)^2$ in $d_\infty$ balls. A real valued function on an interval is strongly convex if there is a $\lambda>0$ such that one can touch the graph of the function from below by a translation of the parabola $y = \lambda x^2$. While smooth convex functions are characterized by the inequality $f''\geq 0$, for $\lambda$-convex functions this inequality turns into $f''\geq 2\lambda$, see Example 4.4.4. in \cite{metric}.

\begin{defn}\label{deflamconv}
A function $f:[0,l]\to \R$ is $\lambda$-convex if $f(t)-\lambda t^2$ is convex for a $\lambda>0$. In this case it is called strongly convex.
\end{defn}

We recall a few facts about the Hessian of $2$-norms which we will use in the proof of the next theorems.

\begin{rem}\label{hessiano}
The Hessian of the $p$-norms in the context of a $C^*$-algebra with a finite trace $\tau$ was studied in \cite{estebancoco,coco}. Let $x, y,z \in \M_{ah}$, let $H_x:\M_{ah}\times \M_{ah}\to \R$ stand for the symmetric bilinear form given by 
$$H_x(y,z) = -2\tau(yz).$$
Denote by $Q_x$ the quadratic form induced by $H_x$. Then (cf. Lemma 4.1 in \cite{estebancoco} and equation (3.1) in \cite{coco}):
\begin{enumerate}
\item $Q_x([y, x])\leq 4\|x\|_\infty^2Q_x(y)$.
\item $Q_x(y)=2\|y\|_2^2$.
\end{enumerate}
In particular $H_x$ is positive definite for any $x\in \M_{ah}$. 

By Remark 3.5 in \cite{pschatten} the same formulas hold if we consider $x,y,z\in B_2(\Hi)_{ah}$ and the symmetric bilinear form $H_x:B_2(\Hi)_{ah}\times B_2(\Hi)_{ah}\to \R$  given by $H_x(y,z) = -2\Tr(yz).$
\end{rem}

The following theorem is an adaptation of Theorem 3.6 in \cite{pschatten} for the case $p=2$, following Remark 3.7 in \cite{pschatten}.

\begin{thm}\label{strong}
Let $u$ be an element of the group $U_2$, let $0<r<\pi/2$, and let $\beta\subseteq B_\infty [u,r]$ be a non constant geodesic. Then the function 
$$f(t)=d_2(u,\beta(t))^2$$ 
satisfies $f''\geq 2\lambda$ for
$$\lambda =c^2\frac{\sin(2r)}{2r} >0,$$ 
where $c>0$ is the speed of the geodesic in the $d_2$ metric. In particular, if $\beta$ has unit speed then $f$ is $\lambda$-convex with 
$$0<\lambda =\frac{\sin(2r)}{2r}<1.$$
\end{thm}

\begin{proof}
We assume that $u=\id$. Let $y,z\in\B_p(\Hi)_{ah}$ such that $\beta(s)=e^ve^{sz}$. Let $x_s=\log(e^ye^{sz})$, and $\gamma_s(t)=e^{tx_s}$. Since $d_\infty (\id,\beta(s))=\|x_s\|_\infty<\pi/2$ the curve $\gamma_s$ is a short geodesic in the $d_2$ metric joining $\id$ and $\beta(s)$ of length $d_2 (\id,\beta(s))=\|x_s\|_2$. Then $f(s)=\|x_s\|_2^2=-\Tr(x_s^2)$, hence
$$f'(s)=-2\Tr(x_s\dot{x}_s)=H_{x_s}(\dot{x}_s,x_s).$$
We denote with $H_a(\cdot,\cdot)$ the Hessian at $a\in\B_p(\Hi)_{ah}$ of the function $\|\cdot\|^2_2$, and $Q_a(\cdot)$ its associated quadratic form. 

Note that $\exp(x_s)=e^ye^{sz}$, so if we differentiate this expression at $s$ we get $d\exp_{x_s}(\dot{x}_s)=e^{-x_s}z$. We have the following formula for the exponential:  
$$e^{-x_s}d\exp_{x_s}(\dot{x}_s)=\int_0^1 e^{-tx_s}\dot{x}_s e^{tx_s} dt,$$
therefore 
$$z=\int_0^1 e^{-tx_s}\dot{x}_s e^{tx_s} dt.$$
Thus
$$\Tr(x_s\dot{x}_s)=\int_0^1 \Tr(x_s e^{-tx_s}\dot{x}_s e^{tx_s})dt=\Tr(zx_s).$$
Hence
$$f''(s)=-2\Tr(\dot{x}_s z)=H_{x_s}(\dot{x}_s,z),$$
and if we put $\delta_s(t)= e^{-tx_s}\dot{x}_s e^{tx_s}$, then
$$f''(s)=\int_0^1 -2\Tr(\delta_s(0)\delta_s(t))dt=\int_0^1 H_{x_s}(\delta_s(0),\delta_s(t))dt.$$
We define $R_s^2=Q_{x_s}(\dot{x}_s)=2\|\dot{x}_s\|_2^2>0$ and note that $\delta_s$ lies in the sphere of radius $R_s$ of the Hilbert space $\B_p(\Hi)_{ah}$ endowed with the inner product $H_{x_s}$. Hence
$$H_{x_s}(\delta_s(0),\delta_s(t))=R_s^2\cos(\alpha_s(t)),$$
where $\alpha_s(t)$ is the angle subtended by $\delta_s(0)$ and $\delta_s(t)$. If $L_0^t(\delta_s)$ is the length in the sphere of the curve $\delta_s$ from $\delta_s(0)$ to $\delta_s(t)$ then
\begin{align*}
R_s\alpha_s(t)&\leq L_0^t(\delta_s)=\int_0^tQ_{x_s}^\frac{1}{2}( e^{-tx_s}[x_s,\dot{x}_s] e^{tx_s})dt\\
&=\int_0^tQ_{x_s}^\frac{1}{2}([x_s,\dot{x}_s])dt=tQ_{x_s}^\frac{1}{2}([x_s,\dot{x}_s]).
\end{align*}
By Remark \ref{hessiano} which states that $Q_a([b,a])\leq 4\|a\|_\infty ^2Q_a(b)$ we see that $Q_{x_s}([x_s,\dot{x}_s])\leq 4\|x_s\|_\infty^2R_s^2$. Hence
$$R_s\alpha_s(t)\leq R_s2t\|x_s\|_\infty<R_s\pi.$$
So
$$\cos(\alpha_s(t))\geq\cos(2t\|x_s\|_\infty)$$
and integrating with respect to the $t$-variable we get
$$f''(s)\geq R_s^2\frac{\sin(2\|x_s\|_\infty)}{2\|x_s\|_\infty}.$$
Since $e^{-x_s}d\exp_{x_s}(\dot{x}_s)=z$ the exponential metric decreasing property in Lemma 3.3 of \cite{pschatten} asserts that $\|\dot{x}_s\|_2\geq \|z\|_2$. Hence $R_s^2=2\|\dot{x}_s\|_2^2\geq 2\|z\|_2^2$. Also $d_\infty (\id,\beta(s))=\|x_s\|_\infty\leq r$, so 
$$\frac{\sin(2\|x_s\|_\infty)}{2\|x_s\|_\infty}\geq \frac{\sin(2r)}{2r},$$
hence
$$f''(s)\geq 2\|z\|_2^2\frac{\sin(2r)}{2r}>0.$$
Note that $\|z\|_2$ is the speed of the geodesic $\beta(s)=e^ye^{sz}$, so the conclusion of the theorem follows.
\end{proof}

In the case of unitaries of finite von Neumann algebras an analogous theorem holds. It is an adaptation of Theorem 3.5 in \cite{finitemeasure}. 

\begin{thm}\label{strongvn}
Let $u$ be an element of the group $U_{\M}$, let $0<r<\pi/2$, and let $\beta\subseteq B_\infty [u,r]$ be a non constant geodesic. Then the function 
$$f(t)=d_2(u,\beta(t))^2$$ 
satisfies $f''\geq \lambda$ for
$$\lambda =2c^2\frac{\sin(2r)}{2r} >0,$$ 
where $c>0$ is the speed of the geodesic in the $d_2$ metric. In particular, if $\beta$ has unit speed then $f$ is $\lambda$-convex with 
$$0<\lambda =\frac{\sin(2r)}{2r}<1.$$
\end{thm}

\begin{proof}
The proof is analogous to the proof of Theorem \ref{strong}. In this case $\delta_s$ lies in a sphere of a pre-Hilbert space, but completeness of the space is not necessary. The exponential metric decreasing property is established in Lemma 3.3 of \cite{finitemeasure}.
\end{proof}

\begin{rem}
The bound $r< \pi/2$ for the strong convexity of $d_2(\id,\beta(t))$ cannot be improved. This follows from calculations with the curve of Example \ref{ejemplomatrices}. The example is about unitary $2\times 2$ matrices, but they can be embedded in $U_2$ or in $U_{\M}$ by choosing matrix units. 
\end{rem}

\begin{rem}
Uniform convexity for metric spaces was defined in \cite{naor}, see Definition 3.2 in that article. A geodesic space $(X,d)$ is $p$-uniformly convex with parameter $0\leq C\leq 2$ if for all $x,y,z\in X$ the constant speed geodesic $\gamma:[0,1]\to X$ joining $y$ and $z$ satisfies
$$d(x,\gamma(t))^p\leq (1-t)d(x,y)^p + td(x,z)^p-\frac{C}{2}t(1-t)d(y, z)^p.$$
Note that a $2$-uniformly convex metric space with parameter $2$ is a $\CAT(0)$ space and a $2$-uniformly convex metric space with parameter $C$ is a space such that for all its elements $x$ the function $d(x,\cdot)^2$ is $\frac{C}{2}$-convex.
Theorems \ref{strong} and \ref{strongvn} assert that in $U_2$ and $U_{\M}$, and for $r<\pi/2$, the balls $B_\infty[\id,r/2]$ are $2$-uniformly convex metric spaces with parameter
$$C_r=2\frac{\sin(2r)}{2r} >0.$$
Note that $C_r\to 2$ as $r\to 0$, that is, the parameter tends to the parameter of a $\CAT(0)$ space. As the ball becomes smaller the curvature tends to the curvature of the tangent space at $\id$ which is a Hilbert space.
\end{rem}

\section{Closed geodesic subsets}\label{sclosedgeo}

We define closed geodesic subsets and give several examples in the case of the full unitary group, the unitary groups of finite von Neumann algebras and unitary Hilbert-Schmidt perturbation of the identity. These examples are balls in the $d_\infty$ metric, geodesic convex hulls, several subgroups, Grassmannians and fixed point sets of actions. In the next section most results are formulated in terms of general closed geodesic subsets, only the particular cases of closed balls and Grassmannians are needed to understand the last section.

\begin{defn}\label{defgeods}
We call $M\subseteq U$ a geodesic subset if for all $u,v\in M$ such that $d_\infty(u,v)<\pi$ the unique geodesic $\gamma_{u,v}$ in $U$ which joins $u$ and $v$ is contained in $M$. We call $M\subseteq U$ a geodesic subset with length parameter $0<l<\pi$ if for all $u,v\in M$ such that $d_\infty(u,v)<l$ the unique geodesic $\gamma_{u,v}$ in $U$ which joins $u$ and $v$ is contained in $M$.
\end{defn}

The orthogonal group is an example of geodesic subset. Given an orthonormal basis $(\xi)_n$ of $\Hi$ we can define a complex conjugation $J:\Hi\to\Hi$ by 
$$J:\sum_n \alpha_n\xi_n\mapsto \sum_n \overline{\alpha_n}\xi_n.$$
The map $J$ is an isometric anti-linear involution. We define 
$$O_{J}=\{u\in U:u=JuJ\}.$$ 

\begin{prop}\label{ortogonalfull}
The orthogonal group $O_{J}$ is a geodesic subset of $U$. 
\end{prop}

\begin{proof}
We assume that one endpoint of the geodesic is $\id$. Let $u\in O_{J}$ be such that $d_\infty(\id,u)<\pi$, then there is a geodesic $\gamma:[0,1]\to U$ given by $e^{tx}$ which joins $\id$ and $u=e^x\in O_{J}$ with $\|x\|_\infty<\pi$. From $J^2=\id$ we have $Je^xJ=e^{JxJ}=e^x$. Since $\|JxJ\|_\infty=\|x\|_\infty<\pi$ we see that $JxJ$ and $x$ are in the domain of injectivity of the exponential map, so that $x=JxJ$. Hence $e^{tx}=e^{tJxJ}=Je^{tx}J$ for $t\in [0,1]$, and we conclude that $\gamma\subseteq O_{J}$.
\end{proof}

Fixed point sets of certain actions are also examples of geodesic subspaces. Consider a group $H$ and two homomorphism $\phi, \rho:H\to U$. We can define an action $\pi$ of $H$ on $U$ by 
$$\pi(h)(u)=\phi(h)u\rho(h)^{-1}.$$
The fixed points of this action are the equivalences of the representations, that is, the $u\in U$ such that $\rho(h)=u^{-1}\phi(h)u$ for all $h$ in $H$. Denote the fixed points by
$$U^H=\{u\in U:\pi(h)(u)=u \mbox{  for all  }h\in H\}.$$

\begin{prop}\label{fixgeodesic}
For two homomorphism $\phi, \rho:H\to U$ the set $U^H$ is a geodesic subset.
\end{prop}

\begin{proof}
Let $u,ue^x\in U^H$ such that $\|x\|_\infty<\pi$. From $\phi(h)u\rho(h)^{-1}=u$ and $\phi(h)ue^x\rho(h)^{-1}=ue^x$ it follows that 
$$e^{\rho(h)x\rho(h)^{-1}}=e^{x}$$
for all $h\in H$. Since $\|\rho(h)x\rho(h)^{-1}\|_\infty=\|x\|_\infty<\pi$ we see that $\rho(h)x\rho(h)^{-1}$ and $x$ are in the domain of injectivity of the exponential, so that $\rho(h)x\rho(h)^{-1}=x$ for all $h\in H$. Therefore, for all $t\in [0,1]$ and all $h\in H$ we get $\rho(h)tx\rho(h)^{-1}=tx$. This means that $ue^{tx}\in U^H$ for all $t\in [0,1]$
\end{proof}

\begin{rem}
The set $U$ is a symmetric space with symmetries $u\cdot v=uv^{-1}u$. The Grassmannians are stable under symmetries since
$$e_p\cdot e_q=e_{e_p q e_p^{-1}}.$$
It is also easy to check that $\pi(h)u\cdot\pi(h)v=\pi(h)(u\cdot v)$ for all $u,v\in U$ and $h\in H$, so the fixed point sets $U^H$ are stable under symmetries. 
\end{rem}

Let $w\in U$ and $r<\pi/2$. If $A\subseteq B_\infty[w,r]$ then for $u,v\in A$ we have $d_\infty(u,v)<\pi$. So the geodesic $\gamma_{u,v}$ is well defined and is contained in $B_\infty[w,r]$ by Theorem \ref{convfull}. We define recursively $A_0=A$ and for $n\in\N$
$$A_n=\{\gamma_{u,v}(t):u,v\in A_{n-1}\mbox{  and  }t\in [0,1]\}.$$
Then $\conv(A)=\cup_{n\in\N}A_n$ is the geodesic convex hull of $A$. It is a geodesic subset of $U$.  

\subsection{Unitary Hilbert-Schmidt group}

\begin{defn}
For $r\geq 0$ and $u\in U_2$ we define 
$$B_{U_2 ,\infty}[u,r]=U_2\cap B_\infty[u,r],$$ 
where $B_\infty[u,r]$ is the closed ball in $(U,d_\infty)$.
\end{defn}

\begin{prop}\label{ball2cg}
For $r<\pi/2$ and $u\in U_2$ the ball $B_{U_2 ,\infty}[u,r]$ is closed and geodesic in $(U_2,d_2)$.
\end{prop}

\begin{proof}
We assume that $u=\id$. By Theorem \ref{geodesicasp} (1) $\gamma_{u,v}\in U_2$ if $u,v\in U_2$ with $d_\infty(u,v)<\pi$. Theorem \ref{convfull} states that $d_\infty(\gamma_{u,v}(t)\id)$ is convex, hence the ball is geodesic. The fact that $B_{U_2 ,\infty}[u,r]$ is closed in $(U_2,d_2)$ follows from the following inequalities 
$$ \frac{2}{\pi} d_\infty(u,v)\leq\|u-v\|_\infty\leq \|u-v\|_2\leq d_2(u,v)$$
for $u,v\in U_2$, see Proposition \ref{basicof} (3) and Theorem \ref{geodesicasp} (3).
\end{proof}

\begin{prop}\label{convcl2cg}
For $r< \pi/2$, $u\in U_2$ and a subset $A\subseteq B_{U_2 ,\infty}[u,r]$ the set 
$$\overline{\conv(A)}^{d_2}$$ 
is closed and geodesic in $(U_2,d_2)$.
\end{prop}

\begin{proof}
The set is closed by definition. Let $(u_n)_n,(v_n)_n\subseteq \conv(A)$ such that $u_n\to u$ and $v_n\to v$ in the $d_2$ metric. Since by Lemma 3.3 in \cite{pschatten} the differentials at each point of
$$\exp:\{x\in \B_2(\Hi)_{ah}: \|u-v\|_\infty\leq 2r\}\to B_{U_2 ,\infty}[\id,2r]$$
are isomorphisms, it follows from the infinite dimensional inverse mapping theorem that this map is a diffeomorphism. It is easy to check that $\gamma_{u_n,v_n}(t)\to \gamma_{u,v}(t)$ for all $t\in [0,1]$ in the $d_2$ metric.
\end{proof}

For an integer $0<m<\infty$ we consider the Grassmannian $\Gr_{m}$ of projections onto $m$ dimensional subspaces of a separable Hilbert space $\Hi$.  We identify the space of projections with a space of symmetries via the map
$$p\mapsto e_p=\id-2p.$$
Each $e_p$ is a unitary operator so that an injection $e:\Gr_{m}\to U_2$ is defined. We have
$$e(\Gr_m)=\{u\in U_2:u=u^*\mbox{  and  }\Tr(\id-u)=2m\}.$$

\begin{prop}\label{grasshilb}
The set $e(\Gr_m)$ is closed and geodesic in $(U_2,d_2)$.
\end{prop}

\begin{proof}
Let $p$ and $q$ be projections onto $m$ dimensional subspaces of $\Hi$ such that $d_\infty(e_p,e_q)<\infty$. Then by Corollary \ref{geoconvref} there is a geodesic $\gamma(t)=e_pe^{tx}$ joining $e_p$ and $e_q=ue^x$. Since $\gamma(t)=e_pe^{tx}= e^{-\frac{1}{2} tx}e_pe^{\frac{1}{2} tx}$ is conjugation by a one parameter group and the trace is invariant by conjugation we see that $\Tr(\id-\gamma(t))$ is well defined and constant. 

To prove that $e(\Gr_n)$ is closed in $(U_2,d_2)$ consider a sequence $(e_{p_n})_n\subseteq e(\Gr_m)$ such that $e_{p_n}\to u\in U_2$ in the $d_2$ distance. Since
$$\|e_{p_n}- u\|_\infty \leq \|e_{p_n}- u\|_2\leq d_2(e_{p_n},u)$$
we see that $p_n\to \frac{1}{2}(\id-u)$ in the uniform norm. From this it follows that $q=\frac{1}{2}(\id-u)$ is a projection. It is also easy to check that it has an $m$-dimensional image, therefore $u=e_q\in e(\Gr_n)$. 
\end{proof}

Let $\A\subseteq \B(\Hi)$ be a $C^*$-subalgebra, and let $U_{\A}$ be its unitary group. We define
$$U_{2,\A}=U_2\cap U_{\A}.$$

\begin{prop}
The group $U_{2,\A}$ is closed and geodesic in $(U_2,d_2)$. 
\end{prop}

\begin{proof}
Since $U_{\A}$ and $U_2$ are geodesic the intersection is geodesic. The group is closed in $U_2$ since $\|u-v\|_\infty\leq \|u-v\|_2\leq d_2(u,v)$ for $u,v\in U_2$.
\end{proof}

Given an isometric anti-linear involution $J:\Hi\to\Hi$ we define the orthogonal group
$$O_{2,J}=\{u\in U_2:u=JuJ\}=U_2\cap O_J .$$ 

\begin{prop}\label{ortogonal}
The orthogonal group $O_{2,J}$ is closed and geodesic geodesic in $(U_2,d_2)$. 
\end{prop}

\begin{proof}
Since $O_J$ and $U_2$ are geodesic the intersection is geodesic. The fact that $O_{2,J}$ is closed follows from $\|JxJ\|_2=\|x\|_2$ for $x\in \B_2(\Hi)_{ah}$ and from the equivalence between convergence in $d_2$ and with the $\|\cdot\|_2$ norm.
\end{proof}

\begin{rem}
The real Grassmannian $e(\Gr_n)\cap O_{2,J}$ is the intersection of closed geodesic spaces so it is closed and geodesic.
\end{rem}

For a group $H$ and two homomorphism $\phi, \rho:H\to U_2$ consider the action as in Proposition \ref{fixgeodesic} and denote by $U_2^H$ the fixed point set. Note that $U_2^H=U_2\cap U^H$.

\begin{prop}
Let $U_2^H$ be a fixed point set as above. It is a closed and geodesic subspace of $(U_2,d_2)$.
\end{prop}

\begin{proof}
It is geodesic since by Proposition \ref{fixgeodesic} $U_2^H=U_2\cap U^H$ is the intersection of two geodesic sets. 
Converge in the $d_2$ metric is equivalent to convergence in the $\|\cdot\|_2$-norm. For all $h\in H$ if $\|u_n-u\|_2\to 0$ then $\|\phi(h)u_n\rho(h)^{-1}-\phi(h)u\rho(h)^{-1}\|_2=\|u_n-u\|_2\to 0$. Hence, if $u_n\in U_2^H$ and $u_n\to u$ in $(U,d_2)$ then $u\in U_2^H$.
\end{proof}

\subsection{Unitary group of a finite von Neumann algebra}

\begin{defn}
For $r\geq 0$ we define 
$$B_{U_{\M},\infty}[u,r]=B_\infty[u,r]\cap U_{\M},$$ 
where $B_\infty[u,r]$ is the closed ball in $(U,d_\infty)$.
\end{defn}

\begin{lem}\label{sotd2}
In $U_{\M}$ convergence in $\SOT$ is equivalent to convergence in the $d_2$ metric.
\end{lem}

\begin{proof}
By Theorem \ref{basicvn} (3) converge in $d_2$ is equivalent to convergence in the $\|\cdot\|_2$-norm. By in \cite[Proposition 5.3, Chapter III]{takesaki} in norm closed balls convergence in $\SOT$ is equivalent to convergence in the norm $\|\cdot\|_2$.
\end{proof}

\begin{prop}\label{ballvncg}
For $r<\pi/2$ and $u\in U_{\M}$ the ball $B_{U_{\M},\infty}[u,r]$ is closed and geodesic in $(U_{\M},d_2)$
\end{prop}

\begin{proof}
We assume that $u=\id$. For if $u,v\in U_{\M}$ with $d_\infty(u,v)<\pi$ we have $\gamma_{u,v}\subseteq U_{\M}$ by properties of the functional calculus. The convexity of the $d_\infty$ distance asserted in Theorem \ref{convfull} implies the geodesic convexity of $B_{U_{\M},\infty}[u,r]$. 
The fact that the ball it is closed in $(U_{\M},d_2)$ is proved as follows. Assume that $(u_n)_n\subseteq B_{U_{\M},\infty}[\id,r]$ is a sequence such that $u_n\to u\in U_{\M}$ in the $d_2$ metric. By Lemma \ref{sotd2} we have $u_n\to u$ in $\SOT$ and by Proposition \ref{ballcwot} the ball $B_\infty[\id,r]$ is $\SOT$ closed, hence $u\in B_\infty[\id,r]$. 
\end{proof}

\begin{prop}\label{convclvn}
For $r< \pi/2$, $u\in U_{\M}$ and a subset $A\subseteq B_{U_{\M} ,\infty}[u,r]$ the set 
$$\overline{\conv(A)}^{d_2}$$ 
is closed and geodesic in $(U_{\M},d_2)$.
\end{prop}

\begin{proof}
By definition the set is closed. Let $(u_n)_n,(v_n)_n\subseteq \conv(A)$ such that $u_n\to u$ and $v_n\to v$ in the $d_2$ metric, which is equivalent to convergence in $\SOT$ by Lemma \ref{sotd2}. By Proposition \ref{contendsot} for a fixed $t\in [0,1]$ the geodesic is $\SOT$ continuous in its endpoints, hence $\gamma_{u_n,v_n}(t)\to \gamma_{u,v}(t)$ in the $d_2$ metric.
\end{proof}

For a finite von Neumann algebra $\M$ and an $s\in [0,1]$ we consider the Grassmannian $\Gr_{\M,s}$ of projections $p\in \M$ such that $\tau(p)=s$. We identify the space of projections with a space of symmetries via the map
$$p=e_p\mapsto \id-2p.$$
Each $e_p$ is a unitary operator so that an injection $e:\Gr_{\M,s}\to U_{\M}$ is defined. We have
$$e(\Gr_{\M,s})=\{u\in U_{\M}:u=u^*\mbox{  and  }\tau (u)=1-2s\}.$$

\begin{prop}\label{grassvn}
The set $e(\Gr_{\M,s})$ is a closed and geodesic subset of $(U_{\M},d_2)$.
\end{prop}

\begin{proof}
Let $p$ and $q$ be projections such that $d_\infty(e_p,e_q)<\pi$. By Corollary \ref{geoconvref} there is a geodesic $\gamma(t)=e_pe^{tx}$ in $U_{\M}$ joining $e_p$ and $e_q=ue^x$ which is given by conjugation, that is, $\gamma(t)=e_pe^{tx}= e^{-\frac{1}{2} tx}e_pe^{\frac{1}{2} tx}$. Since the trace is invariant by conjugation we have $\tau(\gamma(t))=1-2s$ for $t\in [0,1]$. The Grassmannian is closed by \cite[Proposition 2.3]{estebanc}, this also follows from Lemma \ref{sotd2} and the continuity in $\SOT$ of the trace $\tau$.
\end{proof}

Let $\Nv \subseteq \M$ be an inclusion of finite von Neumann algebras, and let $U_{\Nv}$ be the unitary group of $\Nv$.  

\begin{prop}\label{ortvn}
The group $U_{\Nv}$ is a closed and geodesic subspace $(U_{\M},d_2)$. 
\end{prop}

\begin{proof}
The set $U_{\Nv}$ is the intersection of two geodesic spaces. The second part follows from Lemma \ref{sotd2} and the fact that $\Nv$ is $\SOT$ closed, since it is a von Neumann algebra.
\end{proof}

Let $\M\subseteq \B(\Hi)$ be a finite von Neumann algebra and $J:\Hi\to\Hi$ an isometric anti-linear involution such that $\tau(JxJ)=\overline{\tau(x)}$ for all $x\in \M$. We define the orthogonal group
$$O_{\M,J}=\{u\in U_{\M}:u=JuJ\}=U_{\M}\cap O_J.$$

\begin{prop}\label{ortogonalvn}
The orthogonal group $O_{\M,J}$ is a closed and geodesic subspace of $(U_{\M},d_2)$. 
\end{prop}

\begin{proof}
Since $O_J$ and $U_{\M}$ are geodesic the intersection is geodesic. From $\|JxJ\|_2=\|x\|_2$ for $x\in \M_{ah}$ and from the equivalence between convergence in $d_2$ and $\|\cdot\|_2$ we conclude that the orthogonal group is closed.
\end{proof}

\begin{rem}
The real Grassmannian $e(\Gr_{\M,s})\cap O_{\M,J}$ is the intersection of closed geodesic spaces so it is closed and geodesic.
\end{rem}

For a group $H$ and two homomorphism $\phi, \rho:H\to U_{\M}$ consider the action of Proposition \ref{fixgeodesic} and denote by $U_{\M}^H$ the fixed point set. Note that $U_{\M}^H=U_{\M}\cap U^H$.

\begin{prop}
Consider a fixed point set $U_{\M}^H$ as above. Then it is closed geodesic subspace of $(U_{\M},d_2)$.
\end{prop}

\begin{proof}
The fixed point set is geodesic since by Proposition \ref{fixgeodesic} $U_{\M}^H=U_{\M}\cap U^H$ is the intersection of two geodesic sets. Converge in the $d_2$ metric is equivalent to convergence in $\SOT$, therefore for all $h\in H$ if $u_n\to u$ in $\SOT$ then $\phi(h)u_n\rho(h)^{-1}\to \phi(h)u\rho(h)^{-1}$ in $\SOT$. we conclude that if $u_n\in U_2^H$ and $u_n\to u$ in $(U,d_2)$ then $u\in U_2^H$.
\end{proof}

\subsection{Special unitary group in finite dimensions}\label{sectsu}

The finite dimensional special unitary group is an example where the length parameter is less than $\pi$. We set
$$SU(\C^n)=\{u\in \C^{n\times n}:u\mbox{  is unitary and  }\det(u)=1\}.$$

\begin{prop}\label{specialunit}
The special unitary group $SU(\C^n)$ is a geodesic subspace of the unitary group $U(\C^n)$ acting on $\C^n$ and it has length parameter $\min\{\frac{2\pi}{n},\pi\}$. 
\end{prop}

\begin{proof}
Let $u,v\in SU(\C^n)$ such that $d_\infty(u,v)<\min\{\frac{2\pi}{n},\pi\}$. Then there is a geodesic $\gamma:[0,1]\to U$ given by $ue^{tx}$ which joins $u$ and $v=ue^x$. Note that $\|x\|_\infty=d_\infty (u,v)<\frac{2\pi}{n}$ and 
$$1=\det(v)=\det(ue^x)=\det(u)\det(e^x)=e^{\Tr(x)}.$$ 
Since $\Tr(x)\in 2\pi i\Z$ and $\|x\|_\infty <\frac{2\pi}{n}$ we see that $\Tr(x)=0$. Hence $\det(ue^{tx})=\det(u)\det(e^{tx})=e^{\Tr(tx)}=1$ for all $t\in [0,1]$ and the proof is complete.
\end{proof}

\section{Fixed point properties of group actions}\label{scirc}

Complete geodesic spaces with $\lambda$-convex squares of distance functions have well defined circumcenters for its bounded subsets, see Proposition 9.2.24 in \cite{metric}. In the context of this article we first define a set $C$ of approximate circumcenters of a set $A$ in the $d_\infty$ metric, and in this set $C$ we find a minimizer of the strongly convex function $f_A(u)=\sup_{a\in A}d_2(u,a)^2$. Therefore we are able to prove fixed point results with optimal bounds on the $d_\infty$ distance, which are much stronger than $d_2$ bounds. We start with the following definition.

\begin{defn}
Given a geodesic subspace $M$ of $U$ and a subset $A\subseteq M$ we define its $d_\infty$-circumradius relative to $M$ as 
$$\radius_M(A) = \inf\{r:\mbox{there is } c\in M \mbox{ such that } A\subseteq B_{\infty}[c,r]\}.$$
\end{defn}

\subsection{Optimal radius for existence of fixed point}

We prove the existence of minimizers of $f_A$ in sets $C$ for subsets of $d_\infty$-circumradius less than $\pi/2$. We first start with the case of the Hilbert-Schmidt group $U_2$. The following definition is similar to \cite[Definition 9.2.17]{metric}.

\begin{defn}\label{deflamconv2}
Let $M\subseteq U_2$ be a closed geodesic space. A function $f:M\to \R$ is $\lambda$-convex for a $\lambda >0$ if for any unit-speed geodesic $\gamma$ in $(M,d_2)$ the function
$t\mapsto f(\gamma(t))-\lambda t^2$ is convex.
\end{defn}

\begin{lem}\label{supremum}
Let $M\subseteq U_2$ be a closed geodesic set, and let $\{f_i\}_{i\in I}$ be a family of $\lambda$-convex functions on $M$ such that for a $c\in\R$ and for all $i\in I$ we have $f_i\leq c$. Then $\sup\{f_i\}_{i\in I}$ is $\lambda$-convex. 
\end{lem}

\begin{proof}
The bound $c$ implies that the supremum is well defined. For a unit speed geodesic $\gamma:[0,s]\to M$ and for all $i\in I$ the function $f_i(\gamma(t))-\lambda t^2$ is convex. Hence
$$\sup\{f_i(\gamma(t))\}_{i\in I}-\lambda t^2=\sup\{f_i(\gamma(t))-\lambda t^2\}_{i\in I}$$ 
is a convex function since it is the supremum of convex functions. 
\end{proof}

The last part of the next proof is similar to \cite[Proposition 9.2.20]{metric}.

\begin{thm}\label{thmcirc2}
Let $M$ be a closed geodesic subspace of $U_2$ and let $A\subseteq (M,d_2)$ be a bounded subset such that $\radius_M(A)<\pi/2$. For $r\in\R$ such that $\radius_M(A)<r<\pi/2$ and 
$$C=M\cap\left(\bigcap_{a\in A}B_\infty[a,r]\right)$$ 
the function $f_A:C\to \R$ given by 
$$f_A(u)=\sup_{a\in A}d_2(u,a)^2$$
has a unique minimizer.
\end{thm}

\begin{proof}
The condition $\radius_M(A)<r$ implies that there is an $m\in M$ with $A\subseteq B_\infty[m,r]$, so the subset $C$ is not empty. By Proposition \ref{ball2cg} the set $\bigcap_{a\in A}B_{U_2,\infty}[a,r]$ is the intersection of closed and geodesic spaces, so it is closed and geodesic. The set $M$ is also closed and geodesic, hence $C$ is closed and geodesic. 

Since $A$ is bounded the function $f_A$ is well defined. By Theorem \ref{strong} for each $a\in A$ the function $u\mapsto d_2(u,a)^2$ is $\frac{\sin(2r)}{2r}$-convex on $B_{U_2,\infty}[a,r]$, hence also $\frac{\sin(2r)}{2r}$-convex on $C$. The function $f_A$ is a supremum of $\lambda$-convex functions with 
$$\lambda=\frac{\sin(2r)}{2r}>0,$$
so Lemma \ref{supremum} implies that it is $\lambda$-convex. Define
$$r_{\inf}=\inf_{u\in C}f_A(u).$$
Let $r_n=r_{\inf}+1/n$ for $n\in\N$, and set 
$$C_{r_n}=f_A^{-1}((-\infty,r_n^2])=C\cap(\bigcap_{a\in A}B_2[a,r_n]).$$ 
These sets are not empty, closed in $(C,d_2)$ and descending in $n$, that is $C_{r_n}\subseteq C_{r_m}$ if $n\geq m$. We now show that their $d_2$-diameter tends to zero, if $u,v\in C_n$ let $w=\gamma_{u,v}(\frac{1}{2})$ be the midpoint of $u$ and $v$. This midpoint is in $C$ since $C$ is geodesic. From Definition \ref{deflamconv2} it follows that 
$$f_A(w)\leq \frac{f_A(u)+f_A(v)}{2}-\frac{\lambda}{2}d_2(u,v)^2.$$
From $f_A(u),f_A(v)\leq r_{\inf}+1/n$ and $f_A(w)\geq r_{\inf}$ a straightforward calculation shows that
$$d_2(u,v)\leq \sqrt{\frac{2}{n\lambda}}.$$
Since $(C,d_2)$ is complete we conclude that 
$$\bigcap_{n\in\N}C_{r_n}=\{u_m\}$$
for an $u_m\in C$, which is the unique minimizer of $f_A$.
\end{proof}

\begin{rem}
Note that $u\mapsto \sqrt{f_A(u)}=\sup_{a\in A}d_2(u,a)$ is the supremum of $1$-Lipschitz functions with respect to $d_2$, so it is $1$-Lipschitz.
\end{rem}

\begin{thm}\label{fixedpoint}
Let a group $G$ act on a closed geodesic subspace $M$ of $U_2$ such that the $d_2$ and $d_\infty$ metrics are invariant for the action. If there is a $v\in M$ such that its orbit $\oo(v)$ is bounded in the $d_2$ metric and such that $\radius_M(\oo(v))<\pi/2$, then  the action has a fixed point. 
\end{thm}

\begin{proof}
Set $A=\oo(v)$ and let $r\in\R$ which satisfies $\radius_M(\oo(v))<r<\pi/2$. We define as in  Theorem \ref{thmcirc2} the set $C$ and the function $f_A:C\to\R$. Since the action is isometric for the $d_\infty$ distance we see that 
$$\bigcap_{a\in \oo(v)}B_\infty[a,r]$$ 
is invariant for the action, so $C$ is also invariant. Since the action is invariant for the $d_2$ distance we see that the function $f_{\oo(v)}$ is invariant. By Theorem \ref{thmcirc2} the function $f_{\oo(v)}$ has a unique minimizer $u_m\in C$. for all $g\in G$ we have $f_{\oo(v)}(u_m)=f_{\oo(v)}(g\cdot u_m)$, therefore $g\cdot u_m=u_m$ for all $g\in G$. Hence $u_m$ is a fixed point of the action.
\end{proof}

We show that the $\pi/2$ bound in this theorem cannot be improved. 

\begin{ex}
Take $M=U(\C)\simeq S^1$ and $A=\{1,-1\}$. This is the orbit of $1$ by left multiplication with the group $A$.  Then $\oo(1)=A\subseteq B_\infty[i,\pi/2]$ but the action has no fixed point. 
To get an example in the infinite dimensional group $U_2$ consider the action of $A$ on $U_2$ given by $a\cdot u=\diag(a,\id)u$, where the diagonal is defined with respect to a decomposition $\Hi=\C\xi\oplus(\C\xi)^{\perp}$.
\end{ex}

We can get analogous results in the case of the unitary group of a finite von Neumann algebra $\M$.

\begin{defn}\label{deflamconv3}
Let $M\subseteq U_{\M}$ be a closed geodesic space. A function $f:M\to \R$ is $\lambda$-convex for a $\lambda >0$ if for any unit-speed geodesic $\gamma$ in $(M,d_2)$ the function
$t\mapsto f(\gamma(t))-\lambda t^2$ is convex.
\end{defn}

\begin{thm}\label{thmcircvn}
Let $M$ be a closed geodesic subspace of $U_{\M}$ and let $A\subseteq (M,d_2)$ be a subset such that $\radius_M(A)<\pi/2$. For $r\in\R$ such that $\radius_M(A)<r<\pi/2$ and 
$$C=M\cap\left(\bigcap_{a\in A}B_\infty[a,r]\right)$$
the function $f_A:C\to \R$ given by 
$$f_A(u)=\sup_{a\in A}d(u,a)^2$$
has a unique minimizer.
\end{thm}

\begin{proof}
The proof is analogous to the proof of Theorem \ref{thmcirc2}. All subsets of $U_{\M}$ have finite diameter by Theorem \ref{basicof} (3) so $f_A$ is well defined. Proposition \ref{ballvncg} implies that $C$ is closed in the $d_2$ metric. 
\end{proof}

\begin{thm}\label{fixedvn}
Let a group $G$ act on a closed geodesic subspace $M$ of $U_{\M}$ such that the $d_2$ and $d_\infty$ metrics are invariant for the action. If there is a $v\in M$ such that its orbit satisfies $\radius_M(\oo(v))<\pi/2$, then  the action has a fixed point. 
\end{thm}

\begin{proof}
The proof is analogous to the proof of Theorem \ref{thmcirc2}. We note again that $(U_{\M},d_2)$ has finite diameter.
\end{proof}

\begin{rem}
In a finite von Neumann algebra consider a descending family of projections $(p_s)_{s\in [0,1]}$ such that $\tau(p_s)=s$. For each $s$ consider the set of unitaries given by $A_s=\{\id,-p_s+(\id-p_s)\}$. Then 
$$A_s\subseteq B_2\left[ip_s+(\id-p_s),\frac{\pi s}{2}\right]\mbox{   and   } A_s\subseteq B_\infty\left[ip_s+(\id-p_s),\frac{\pi}{2}\right],$$ 
so $A_s$ has arbitrarily small diameter in the $d_2$ metric, and diameter less than $\pi/2$ with respect to the $d_\infty$ distance. It is the orbit of $1$ by left multiplication with the group $A_s$. This action has no fixed points, hence the radius bound in the theorem is optimal. Also observe that a condition with respect to the circumradius in the $d_2$ distance is not enough to guarantee the existence of a fixed point.   
\end{rem}

We can also obtain a similar result when the length parameter is less than $\pi$, such as the special unitary group $SU(\C^n)$ considered in Proposition \ref{specialunit}. We omit the proof of the following proposition.

\begin{prop}\label{fixedsu}
Let a group $G$ act on $SU(\C^n)$ such that the $d_2$ and $d_\infty$ metrics are invariant. If $\radius_{SU(\C^n)}(\oo(v))< \min\{\frac{2\pi}{n},\pi\}/2$ for a $v\in SU(\C^n)$, then  the action has a fixed point. 
\end{prop}

\begin{rem}
On the sets $B_\infty[u,r]$ with $u\in U_2$ and $r<\pi/4$ the center of mass can be defined, see \cite{naor} and the references therein. For a probability measure $\mu$ on $B_\infty[u,r]$ we define
$$u\mapsto \int_C d_2(u,v)^2 d\mu (v).$$
If the functions $u\mapsto d_2(u,v)^2$ are $\lambda$-convex then the integral above is a $\lambda$-convex function and its unique minimizer is called the center of mas.
\end{rem}

\subsection{Existence of equivalences of representations and invariant subspaces}\label{srigid}

Rigidity problems ask under what conditions on two group homomorphisms $\phi,\rho:H\to G$ there is a $g\in G$ such that $\phi(h)=g\rho(h)g^{-1}$. Local rigidity results assert that if $\phi$ and $\rho$ are close in some sense, so that $\{\phi(h)\rho(h)^{-1}:h\in H\}$ is small set, then a $g$ giving an equivalence between the two representations exists. Analogous statements for invariant subspaces of representations are also considered. In this section we establish optimal bounds for rigidity in terms of the circumradius in the $d_\infty$ distance. We start with the case of the Hilbert-Schmidt group $U_2$.

\begin{thm}\label{rigrep}
Let $\phi,\rho:H\to G$ be two homomorphisms into a closed geodesic subgroup $G \subseteq U_2$. If there is an $u\in G$ such that 
$$\radius_G(\{\phi(h)u\rho(h)^{-1}:h\in H\})<\pi/2$$ 
then there is a $g\in G$ such that 
$$\phi(h)=g\rho(h)g^{-1}$$
for all $h\in H$.
\end{thm} 

\begin{proof}
The action of $H$ on $G$ given by
$$h\mapsto (g\mapsto \phi(h)g\rho(h)^{-1})$$
is isometric in the $d_{\infty}$ and $d_2$ metrics by translation invariance. We have $\oo(u)=\{\phi(h)u\rho(h)^{-1}:h\in H\}$ and $\radius_G(\oo(u))<\pi/2$. By Theorem \ref{fixedpoint} the action has a fixed point $g\in G$, that is, $\phi(h)g\rho(h)^{-1}=g$ for all $h\in H$. Hence $\phi(h)=g\rho(h)g^{-1}$ for all $h\in H$. 
\end{proof}

We next show that the $\pi/2$ bound on the circumradius is optimal.

\begin{ex}
Let $H$ be the two element group and $G=U_1\simeq S^1$ the unitary group on one dimensional space. If $\phi$ is the trivial representation and $\rho$ is the non trivial representation, then $\{\phi(h)\rho(h)^{-1}:h\in H\}=\{-1,1\}\subseteq S^1$. Note that $\radius_G(\{-1,1\})=\pi/2$. The group $G$ can be embedded in $U_2$.
\end{ex}

\begin{rem}
In \cite{svarc} the authors constructed a family of representations $\rho_t:G\to U$ with $t\in\R$ such that 
$$\sup_{g\in G}\|\rho_{t_1}(g)-\rho_{t_2}(g)\|_\infty =\sup_{g\in G}\|\rho_{t_2}(g)^{-1}\rho_{t_1}(g)-\id\|_\infty\to 0$$
as $t_1,t_2\to t_0\in \R$, and such that $\rho_{t_1}$ is not equivalent to $\rho_{t_2}$ if $t_1\neq t_2$. Hence, analogous fixed point results are not possible in the context of the full unitary group. Otherwise we could get a fixed point of the action
$$g\mapsto (u\mapsto \rho_{t}(g)u\rho_{0}(g))$$
which makes the representations equivalent. In particular, there is no bi-invariant distance on $U$ which is equivalent to the $d_\infty$ distance and which is uniformly convex on sufficiently small balls.
\end{rem}

\begin{rem}
Another approach is based on the construction of the Riemannian center of mass or Frechet-Karcher mean on compact Lie groups. The study of these means began in \cites{karcher,karcher3} and was applied to some rigidity problems, that is, metric conditions for the equivalence of representations. These centers of mass are generalizations of the Frechet center of mass and are defined for probability measures. Local rigidity results were obtained when $G$ is a group of unitaries in Proposition 4.4 1) and Lemma 2.6 of \cite{ulam} using operator algebra techniques.
\end{rem}

On Grassmannians we can obtain a similar fixed point theorem. 

\begin{lem}\label{gbound}
The Grassmannian $e(\Gr_m)$ is bounded in $(U_2,d_2)$.
\end{lem}

\begin{proof}
Let $p$ and $q$ be projections onto $n$ dimensional subspaces $\Hi_1$ and $\Hi_2$ of $\Hi$.  Since $e_p$ is the identity on $\Hi_1^\perp$ and $e_q$ is the identity on $\Hi_2^\perp$ it follows that both are symmetries in the finite dimensional space $\Hi_1+\Hi_2$ of dimension less than $2n$. In $U(\Hi_1+\Hi_2)$ the diameter in the $d_2$ metric is finite.
\end{proof}

\begin{lem}\label{iactionfull}
The full unitary group $U$ acts isometrically by conjugation on $(U_2,d_2)$.
\end{lem}

\begin{proof}
The norm $\|\cdot\|_2$ is invariant by conjugation, that is, $\|x\|_2=\|uxu^{-1}\|_2$ for $x\in\B_2(\Hi)_{ah}$ and $u\in U$. The proposition follows since the length functional $\Le_2$ is defined in terms of $\|\cdot\|_2$. 
\end{proof}

\begin{thm}\label{riggrass}
Let $H\subseteq U$ be a subgroup which acts by conjugation on $e(\Gr_{m})$, that is
$$h\mapsto (e_p\mapsto he_ph^{-1}).$$
If there is a projection $q_1$ onto an $m$-dimensional subspace such that 
$$\radius_{e(\Gr_m)}(\oo(e_{q_1})<\pi/2,$$ 
then there is a projection $q$ onto an $m$-dimensional subspace such that $hq=qh$ for all $h\in H$. 
\end{thm}

\begin{proof}
The action is isometric for the $d_{\infty}$ and $d_2$ metrics by Lemma \ref{iactionfull}. The orbit of $q_1$ is bounded in the $d_2$ metric since the Grassmannian is bounded by Lemma \ref{gbound}. Therefore, with $M=\Gr_m$, Theorem \ref{fixedpoint} asserts that the action has fixed point $e_q$ in $e(\Gr_{m})$. Hence, the projection $q$ satisfies 
$$he_qh^{-1}=h(\id-2q)h^{-1}=\id-2q$$
for all $h\in H$, so the conclusion of the theorem follows.
\end{proof}

The $\pi/2$ bound on the circumradius is optimal.

\begin{ex}
Consider the Grassmannian $\Gr_{1}(\C^2)$ of one dimensional subspaces in $\C^2$ and take the projections $p$ and $q$ onto the second and first coordinates of $\C^2$. Hence

$e_p=
 \left(\begin{array}{cc}
1 & 0 \\
\\
0 & -1
\end{array}\right)
$
,\quad
$
e_q=
 \left(\begin{array}{cc}
-1 & 0 \\
\\
0 & 1
\end{array}\right)
$
\quad and we define \quad
$
x=
 \left(\begin{array}{cc}
0 & -1 \\
\\
1 & 0
\end{array}\right).
$

\noindent
Then  
$$\gamma(t)=e_pe^{tx}=e^{-\frac{1}{2} tx}e_pe^{\frac{1}{2}tx}$$ 
for $t\in [0,\pi]$ is a curve from $e_p$ to $e_q$. It has speed $\|x\|_\infty =1$ in the $d_\infty$ metric and total length $\pi$. If we consider $\gamma(\pi/2)\in e(\Gr_1(\C^2)$ then $d_\infty(\gamma(\pi/2),e_p)= \pi/2$ and $d_\infty(\gamma(\pi/2),e_q)= \pi/2$, therefore 
$$\{e_p,e_q\}\subseteq B_\infty[\gamma(\pi/2),\pi/2],$$
which implies $\radius_{e(\Gr_1(\C^2)}\leq \pi/2$. We define $H$ as the group consisting of maps 
$$(x_1,x_2)\mapsto (s_1x_{d(1)},s_2x_{d(2)})$$ 
for signs $s_1,s_2$ and permutation $d$ of $\{1,2\}$. The orbit of $e_p$ is $\{e_p,e_q\}$ and this group action has no fixed points. To get an example in the infinite dimensional context consider $\Hi'=\C^2\oplus\Hi$, $e'_p=\diag(e_p,\id)$ and $H'=\{\diag(h,u):h\in H,u\in U(\Hi)\}$.
\end{ex}

We can obtain similar results in the case of finite von Neumann algebras. We omit the proofs since they are similar.

\begin{thm}
Let $\phi,\rho:H\to G$ be two homomorphisms into a closed geodesic subgroup $G \subseteq U_{\M}$. If there is an $u\in G$ such that $\radius_{G}(\{\phi(h)u\rho(h)^{-1}:h\in H\})<\pi/2$ then there is a $g\in G$ such that 
$$\phi(h)=g\rho(h)g^{-1}$$
for all $h\in H$.
\end{thm} 

\begin{thm}
Let $H\subseteq U_{\M}$ be a subgroup which acts by conjugation on $e(\Gr_{\M,s})$, that is
$$h\mapsto (e_p\mapsto he_ph^{-1}).$$
If there is an $e_{q_1}\in e(\Gr_{\M,s})$ such that 
$$\radius_{e(\Gr_m)}(\oo(e_{q_1}))<\pi/2,$$ 
then there is a projection $q\in \M$ with  $\tau(q)=s$ such that $hq=qh$ for all $h\in H$. 
\end{thm}

\begin{rem}
The bounds in this context are also optimal. If we choose matrix units then we can define an embedding $M_2(\C)\subseteq \M$ and the examples follow from the examples in $U_{M_2(\C)}=U(\C^2)$. 
\end{rem}






\noindent

\begin{thebibliography}{XX}

\bibitem[A14]{esteban} E. Andruchow, \textit{The Grassmann manifold of a Hilbert space.} Proceedings of the XIIth ``Dr. Antonio A. R. Monteiro'' Congress, 41--55, Actas Congr.``Dr. Antonio A. R. Monteiro'', Univ. Nac. del Sur, Bahía Blanca, 2014. 

\bibitem[ACL10]{finitemeasure} E. Andruchow, E. Chiumiento, G. Larotonda, \textit{Homogeneous manifolds from noncommutative measure spaces.} J. Math. Anal. Appl. 365 (2010), no. 2, 541--558.

\bibitem[AL10]{fredholm} E. Andruchow, G. Larotonda, \textit{The rectifiable distance in the unitary Fredholm group.} Studia Math. 196 (2010), no. 2, 151--178.

\bibitem[ALR10]{pschatten} E. Andruchow, G. Larotonda, L. Recht, \textit{Finsler geometry and actions of the $p$-Schatten unitary groups.} Trans. Amer. Math. Soc. 362 (2010), no. 1, 319--344.


\bibitem[AR06]{estebanc} E. Andruchow, L. Recht, \textit{Grassmannians of a finite algebra in the strong operator topology.} Internat. J. Math. 17 (2006), no. 4, 477--491.

\bibitem[AR08]{estebancoco} E. Andruchow, L. Recht, \textit{Geometry of unitaries in a finite algebra: variation formulas and convexity.} Internat. J. Math. 19 (2008), no. 10, 1223--1246.




\bibitem[BBI01]{metric} D. Burago, Y. Burago, S. Ivanov, \textit{A course in metric geometry.} Graduate Studies in Mathematics, 33. American Mathematical Society, Providence, RI, 2001.

\bibitem[CPR94]{corach} G. Corach, H. Porta, L. Recht, \textit{ Convexity of the geodesic distance on spaces of positive operators.} Illinois J. Math. 38 (1994), no. 1, 87--94.



\bibitem[BOT13]{ulam} M. Burger, N. Ozawa, A. Thom, \textit{On Ulam stability.} Israel J. Math. 193 (2013), no. 1, 109--129. 

\bibitem[GK73]{karcher} K. Grove, H. Karcher, \textit{How to conjugate $C\sp{1}$-close group actions.} Math. Z. 132 (1973), 11--20.

\bibitem[GKR74]{karcher3} K. Grove, H. Karcher, E. A. Ruh, \textit{Jacobi fields and Finsler metrics on compact Lie groups with an application to differentiable pinching problems.} Math. Ann. 211 (1974), 7--21.

\bibitem[dH72]{harpe} P. de la Harpe, \textit{Classical Banach-Lie algebras and Banach-Lie groups of operators in Hilbert space.} Lecture Notes in Mathematics, Vol. 285. Springer-Verlag, Berlin-New York, 1972. 


\bibitem[M22]{miglioli}	M. Miglioli, \textit{Circumcenters in Finsler unitary groups},  arXiv:2108.05031.


\bibitem[MR00]{coco} L. E. Mata-Lorenzo, L. Recht, \textit{Convexity properties of ${\rm Tr}[(a^*a)^n]$.} Linear Algebra Appl. 315 (2000), no. 1-3, 25--38. 

\bibitem[NS11]{naor} A. Naor, L. Silberman, \textit{Poincar\'e inequalities, embeddings, and wild groups.} Compos. Math. 147 (2011), no. 5, 1546--1572. 

\bibitem[PR87]{portarecht} H. Porta, L. Recht, \textit{Minimality of geodesics in Grassmann manifolds}, Proc. Amer. Math. Soc. 100 (1987), 464-466.


\bibitem[PS86]{svarc} T. Pytlik, R. Szwarc, \textit{An analytic family of uniformly bounded representations of free groups.} Acta Math. 157 (1986), no. 3-4, 287--309.

\bibitem[T02]{takesaki} M. Takesaki, \textit{Theory of operator algebras. I.} Reprint of the first (1979) edition. Encyclopaedia of Mathematical Sciences, 124. Operator Algebras and Non-commutative Geometry, 5. Springer-Verlag, Berlin, 2002. xx+415 pp. ISBN: 3-540-42248-X

\end{thebibliography}
\end{document}